\documentclass[12pt]{amsart}
\usepackage{amsmath,amsfonts,amssymb,amsthm}
\usepackage{thmtools}
\usepackage{graphicx,color}

\usepackage{forest}
\usepackage{float}

\usepackage[hypertexnames = false]{hyperref}

\usepackage{cleveref}

\newtheorem{theorem}{Theorem}[section]
\newtheorem{definition}[theorem]{Definition}

\newtheorem{lemma}[theorem]{Lemma}
\newtheorem{corollary}[theorem]{Corollary}
\newtheorem{proposition}[theorem]{Proposition}

\newtheorem{remark}[theorem]{Remark}
\newtheorem{example}[theorem]{Example}

\newcommand{\N}{\mathbb{N}}
\newcommand{\Z}{\mathbb{Z}}
\newcommand{\R}{\mathbb{R}}
\newcommand{\Q}{\mathbb{Q}}
\newcommand{\C}{\mathbb{C}}

\newcommand{\abs}[1]{\left\lvert#1\right\rvert}

\def\Diff{\operatorname{Diff}}
\def\Leb{\operatorname{Leb}}
\def\Disc{\operatorname{Disc}}

\title[Complex Lagrange spectrum with real part 1/2]{Lagrange spectrum for Diophantine approximations of complex numbers with real part equal to one half}

\author[H. Cheng]{Hao Cheng}
\address{IMPA, Estrada Dona Castorina 110, 22460-320, Rio de Janeiro, Brazil}
\email{hao.cheng@impa.br}
\thanks{The first author was partially supported by CNPq.}

\author[H. Erazo]{Harold Erazo}
\address{SUSTech International Center for Mathematics, Shenzhen, Guangdong, People’s Republic of China}
\email{harold.eraz@gmail.com}
\thanks{The second author was partially supported by CAPES and FAPERJ}

\author[C. G. Moreira]{Carlos Gustavo Moreira}
\address{Carlos Gustavo Moreira: SUSTech International Center for Mathematics, Shenzhen, Guangdong, People’s Republic of China and IMPA, Estrada Dona Castorina 110, 22460-320, Rio de Janeiro, Brazil}
\email{gugu@impa.br}
\thanks{The third author was partially supported by CNPq and FAPERJ}

\author[T. Vasconcelos]{Thiago Vasconcelos}
\address{IMPA, Estrada Dona Castorina 110, 22460-320, Rio de Janeiro, Brazil}
\email{thiago.greco@impa.br}
\thanks{The fourth author was partially supported by CNPq}

\date{\today}


\date{\today}


\begin{document}

\begin{abstract}
    Motivated by a theorem of A. Schmidt on the part of the complex Lagrange spectrum below $2$, we study the restricted Lagrange spectrum $L_{\frac{1}{2}+i\R}$ arising from the approximation of complex numbers of the form $\frac{1}{2}+i\alpha$, $\alpha\in\R\setminus\Q$ by Gaussian rationals $p/q$ with $p,q\in\Z[i]$, $q\neq 0$. We show that this spectrum admits a description in terms of a dynamical spectrum associated with a real horseshoe. As a consequence, we obtain several fractal properties of $L_{\frac{1}{2}+i\R}$, such as continuity of the dimension function $t\mapsto\dim_H(L_{\frac{1}{2}+i\R}\cap(-\infty,t))$.  

    We also prove that the set of complex numbers $z$ satisfying
    \begin{equation*}
        \abs{z-\frac{p}{q}}\geq\frac{1}{2|q|^2}, \quad\text{for all } p,q\in\Z[i], q\neq 0,
    \end{equation*}
    is uncountable. In fact, we show that this inequality holds for every complex number of the form $z=\frac{1}{2}(1+i\theta)$ where $\theta\in\R\setminus\Q$ is a root of one of Schmidt's $C$-minimal forms. 
\end{abstract}

\maketitle

\section{Introduction}

One of the earliest systematic studies of Diophantine approximation in the complex plane by elements of $\Q(i)=\{a+bi:a,b\in\Q\}$ was carried out by Adolf Hurwitz through his theory of complex continued fractions \cite{Hur}. He used an analogue of the continued fraction expansion for complex numbers which are now called Hurwitz continued fractions. More details on this algorithm can be found in \cite{Bug}. Hurwitz proved the following version of Dirichlet's theorem for Diophantine approximations of complex numbers:

\begin{theorem}[Hurwitz]\label{thm:Hurwitz}
    Let $\xi\in \C\setminus \Q(i)$. Then the inequality $|\xi-\frac{p}{q}|<\frac{1}{|q|^2}$ has infinitely many solutions $(p,q)\in\Z[i]\times\Z[i]$ with $q\neq 0$.
\end{theorem}

\Cref{thm:Hurwitz} naturally leads to the study of the quality of approximation of complex irrational numbers by Gaussian rationals. For $\xi\in\Q(i)$, define its \emph{Lagrange constant} by
\begin{equation*}
    k(\xi):=\underset{p,q\in\Z\left[i\right]}{\underset{|p|,|q|\rightarrow \infty}{\limsup}}(|q||q\xi-p|)^{-1}.
\end{equation*}
The \emph{complex Lagrange spectrum}, also known as the \emph{Hurwitz spectrum}, is then defined by
\begin{equation*}
    L_{\Q(i)}:=\{k(\xi):\xi\in \C\setminus \Q(i), k(\xi)<\infty\}.
\end{equation*}

The first result on this spectrum was obtained by Ford \cite{Ford}, who improved \Cref{thm:Hurwitz} by proving that the minimum of $L_{\Q(i)}$ is $\sqrt{3}=k(\frac{1}{2}+i\frac{\sqrt{3}}{2})$. However, the study of $L_{\Q(i)}$ via Hurwitz continued fractions is considerably more delicate than in the real case, since Hurwitz continued fractions lack several of the good properties enjoyed by regular continued fractions. In order to investigate the fine structure of $L_{\Q(i)}$, A. L. Schmidt \cite{Sch} introduced a different continued fraction algorithm based on Farey sets.

Using this approach, Schmidt established in \cite[Theorem 6.7]{Sch} an analogue of Markov's theorem for the complex Lagrange spectrum. He completely characterized the initial portion of $L_{\Q(i)}$ up to its first accumulation point. Indeed, he showed that $L_{\Q(i)}\cap(-\infty,2)$ is a discrete set whose elements can be described in terms of the Vulakh's Diophantine equation $y_1^2+y_2^2+2x^2=4y_1y_2x$, whose solutions can be organized in a ternary tree, analogous to the binary tree of solutions of Markov's equation $x^2+y^2+z^2=3xyz$. In particular, $2$ is the first accumulation point of $L_{\Q(i)}$. On the other hand, he also proved that there are uncountably many equivalence classes of complex numbers $\xi$ satisfying $k(\xi)=2$. 

We say that two complex numbers $\xi$ and $\eta$ are \emph{equivalent} if there are Gaussian integers $a,b,c,d\in\Z[i]$ with $ad-bc\in\{\pm 1,\pm i\}$ such that $\eta=(a\xi +b)/(c\xi +d)$. Notice that equivalent numbers have the same Lagrange constant, i.e., $k(\xi)=k(\eta)$.

\begin{theorem}[A. Schmidt]\label{thm:schmidt}
Let $\mathcal{N}(\Lambda):=\{1, 5, 29, 65, 169, 349,\dots\}$ be the sequence of positive integers arising as the values of $x$ in the Diophantine equation $y_1^2+y_2^2+2x^2=4y_1y_2x$. For each $\Lambda\in\mathcal{N}(\Lambda)$, there exists a complex quadratic irrational $\xi_{\Lambda}\in\Q(i\sqrt{4\Lambda^2-1})$ such that the following holds. If $\xi\in \C\setminus \Q(i)$ satisfies $k(\xi)<2$, then either
\begin{itemize}
    \item $k(\xi)=\sqrt{\frac{3}{5}\sqrt{41}}$, in which case $\xi$ is equivalent to one of the four numbers $\frac{1}{6}(3\pm\sqrt{-15\pm 12i})$;
    \item or $\xi$ is equivalent to $\xi_\Lambda$ for some $\Lambda\in\mathcal{N}(\Lambda)$.
\end{itemize}
Moreover, each $\xi_\Lambda$ is equivalent to a number of the form $\frac{1}{2}(1+i\theta_\Lambda)$, where $\theta_\Lambda\in\R\setminus\Q$ is a root of a form $g_\Lambda$ defined on \eqref{eq:forms} and $k(\xi_\Lambda)=\sqrt{4-\Lambda^{-2}}$. Consequently,
\begin{align*}
    L_{\Q(i)}\cap (-\infty,2)&=\left\{\sqrt{4-\Lambda^{-2}}:\Lambda\in\mathcal{N}(\Lambda)\right\}\cup\left\{\sqrt{\frac{3}{5}\sqrt{41}}\right\}
\end{align*}
\end{theorem}

The forms $g_\Lambda$ appearing in \Cref{thm:schmidt} are a family of binary quadratic rational forms introduced by Schmidt in \cite{Sch} and studied further in \cite{SchmidtCforms}. They play an important role in some of the results of this paper, and their definition and main properties will be recalled in \Cref{sec:binary_forms}.

Analogous to the real case, Schmidt also studied the infimum of complex binary quadratic forms. More precisely, given
\begin{equation*}
    f(x,y)=Ax^2+Bxy+Cy^2, \quad A,B,C\in\C, \quad B^2-4AC\neq 0,
\end{equation*}
he defined
\begin{equation*}
    m(f):=\frac{\sqrt{|B^2-4AC|}}{\inf_{p,q\in\Z[i]^2\setminus(0,0)}|f(p,q)|}.
\end{equation*}

The \emph{complex Markov spectrum} is the set
\begin{equation*}
    \mathcal{M}_{\Q(i)} = \left\{m(f):f(x,y)=Ax^2+Bxy+Cy^2,A,B,C\in\C,B^2-4AC\neq 0\right\}.
\end{equation*}

In \cite[Theorem 6.6]{Sch}, Schmidt characterized the initial portion of this spectrum and showed, in particular, that
\begin{equation*}
    L_{\Q(i)}\cap(-\infty,2)
    =
    \mathcal{M}_{\Q(i)}\cap(-\infty,2).
\end{equation*}
Moreover, he determined all forms giving rise to this discrete part of the complex Markov spectrum and showed that they can be constructed from the $C$--minimal forms $g_\Lambda$.

\subsection{Approximations of numbers whose real parts is one half}
In his proof of \Cref{thm:schmidt}, Schmidt observed that, with the exception of the value $\sqrt{\frac{3}{5}\sqrt{41}}$, every element of the discrete part of $L_{\Q(i)}$ arises as the approximation constant of a complex number of the form $\frac{1}{2}
+i\alpha$ with $\alpha\in \R\setminus \Q$. Moreover, by restricting his algorithm to complex numbers whose real part is $\frac{1}{2}$, he obtained a theory exhibiting striking similarities with the classical theory of regular continued fractions. Using Moreira's theory of dynamically defined Lagrange spectra \cite{LMV}, it will follow that the spectrum
\begin{equation*}
    L_{\frac{1}{2}+i\R}:=\left\{k\left(\frac{1}{2}+i\alpha\right):\alpha\in\R\setminus \Q, k\left(\frac{1}{2}+i\alpha\right)<\infty\right\},
\end{equation*}
can be realized as a dynamically defined Lagrange spectrum associated with a conservative horseshoe on a surface, exactly as in the real setting. 

Let $\xi=\frac{1}{2}+i\alpha$ with $\alpha\in \R\setminus \Q$. The restriction of Schmidt's algorithm for this class of numbers is related to what he calls the $C$-regular and $C$-dually regular continued fraction expansions. The $C$-regular expansion of $\alpha$ is the representation 
$$\alpha=2a_0-1+\frac{2\varepsilon_1}{2a_1+\frac{2}{2a_2+\frac{2\varepsilon_3}{2a_3+\frac{2}{2a_4+\frac{2\varepsilon_4}{\ddots}}}}},$$ 

where $a_n\in \N$ and $\varepsilon_{2n-1}\in \{1,-1\}$.
Similarly, the $C$-dually regular expansion represents $\alpha$ as $$\alpha=2b_0-1+\frac{2}{2b_1+\frac{2\varepsilon_2}{2b_2+\frac{2}{2a_3+\frac{2\varepsilon_4}{2a_4+\frac{2}{\ddots}}}}}.$$ 

Schmidt studied the approximation of $\xi$ by complex numbers of the form $
\frac{1}{2}(1+i\frac{p_n}{q_n})$, where the $\frac{p_n}{q_n}$ are convergents of the $C$-regular or $C$-dually regular continued fraction expansion of $\alpha$. We refer the reader to \cite[Section 4]{Sch} for the details of this construction.

In the present paper, we adopt a different approach. Since classical continued fractions are more familiar and better suited to our purposes, we shall work directly with the classical continued fraction expansion of $\alpha$. This will be done with the following lemma:

\begin{lemma}\label{lem:A}
Let $p,q\in \Z[i]$ and $\alpha\in \R\setminus\Q$. If 
\begin{equation*}
    \abs{\frac{1}{2}+i\alpha-\frac{p}{q}}\leq\frac{1}{2\abs{q}^2},
\end{equation*}
then $\Re(p/q)=1/2$. Moreover, for any such approximation with $\Re(p/q)=1/2$, if $(p_m/q_m)_{m\geq 0}$ are the convergents of the continued fraction of $\alpha$,  then one of the following holds for some $n\geq 0$:
\begin{itemize}
    \item $\Im(p/q)=\frac{p_{n+1}-p_n}{q_{n+1}-q_n}$ and $q_{n+1}-q_n\equiv 2\pmod{4}$; or
    \item $\Im(p/q)=\frac{p_{n}+p_{n-1}}{q_{n}+q_{n-1}}$ and $q_n+q_{n-1}\equiv 2\pmod{4}$; or
    \item $\Im(p/q)=\frac{p_n}{q_n}$.
\end{itemize}
\end{lemma}

Schmidt proved an analogous statement using the convergents of the $C$-regular and $C$-dually regular continued fraction expansions. Together with his characterization of the discrete part of $L_{\frac{1}{2}+i\R}$, the above lemma shows that the non-discrete portion of the restricted spectrum $L_{\frac{1}{2}+i\R}\cap(2,\infty)$ can be studied through the classical continued fraction expansion of the imaginary part.

Throughout the paper, $\dim_H$ denotes Hausdorff dimension. We shall prove the following results concerning $L_{\frac{1}{2}+i\R}$.

\begin{theorem}\label{thm:main}
Let $d_0(t):=\dim_H\left(L_{\frac{1}{2}+i\R}\cap (-\infty,t)\right)$ for $t\in\R$. Then we have that: 
    \begin{enumerate}
        \item\label{itm:thmcont} The function $d_0$ is continuous.
        \item\label{itm:thmcont2} $d_0(t)=\overline{\dim_{B}}\left(L_{\frac{1}{2}+i\R}\cap (-\infty,t)\right)$ for $t\in\R$.
        \item $\max\{t\in\R:d_0(t)=0\}=2$.
        \item The spectrum $L_{\frac{1}{2}+i\R}$ is a closed set.
        \item The set of accumulation points $L_{\frac{1}{2}+i\R}^\prime$ form a perfect set, that is,  $L_{\frac{1}{2}+i\R}^{\prime\prime}=L_{\frac{1}{2}+i\R}^\prime$.
    \end{enumerate}
\end{theorem}

We then have the immediate corollary for $L_{\Q(i)}$:

\begin{corollary}\label{cor:dimension}
    Let $d(t):=\dim_H(L_{\Q(i)}\cap (-\infty,t))$. Then we have that $d(2+\varepsilon)>0$ for every $\varepsilon>0$.
\end{corollary}

To prove this result, we will use a completely different approach from Schmidt for a matter of convenience. Our construction is based on the classical continued fraction algorithm, via the \Cref{lem:A}.

\subsection{Poorly approximable complex numbers}

We say that a complex number $z$ is \emph{poorly approximable} if
\begin{equation}\label{eq:1/2|q|^2}
    \abs{z-\frac{p}{q}}\geq\frac{1}{2|q|^2}, \quad\text{for all } p,q\in\Z[i], q\neq 0.
\end{equation}

Gurwood \cite{Gurwood} introduced and characterized the real poorly approximable numbers, where the constant $\frac{1}{2}$ in \eqref{eq:1/2|q|^2} is replaced by $\frac{1}{3}$.

It follows from \Cref{thm:schmidt} that, if $\theta_\Lambda$ is a root of one of the forms $g_\Lambda$ defined in \eqref{eq:forms}, then the number $z=\frac{1}{2}(1+i\theta_\Lambda)$ satisfies \eqref{eq:1/2|q|^2} with at most finitely many exceptions. We shall prove that, in fact, no exceptions occur.

In the real setting, Harcos \cite{Harcos} and Florek \cite{Florek} observed that the roots of the classical Markov forms enjoy a stronger property than merely having Lagrange constants in the discrete part of the spectrum: they are poorly approximable. We prove an analogous result for the roots of Schmidt's $C$-minimal forms. These are binary quadratic real forms satisfying properties closely analogous to those of the classical Markov forms and were introduced by Schmidt in his study of the complex Lagrange spectrum. Their definition is recalled in detail in \Cref{sec:binary_forms}.

\begin{theorem}\label{thm:poorly_approximable}
    Let $z=\frac{1}{2}(1+i\theta)$ where $\theta$ is a root of one of the forms $g_\Lambda$ or $h_M$ defined on \eqref{eq:forms}. Then
    \begin{equation*}
        \abs{z-\frac{p}{q}}\geq\frac{1}{2|q|^2}, \quad\text{for all } p,q\in\Z[i], q\neq 0.
    \end{equation*}
\end{theorem}
In particular, an application of the argument used in \cite[Chapter II, Lemma 14]{Cassels} yields the following result.
\begin{theorem}\label{thm:uncountably_many_for_c=2}
    There exist uncountably many complex numbers $z=\frac{1}{2}(1+i\theta)$ such that
    \begin{equation*}
        \abs{z-\frac{p}{q}}\geq\frac{1}{2|q|^2}, \quad\text{for all } p,q\in\Z[i], q\neq 0.
    \end{equation*}
    More precisely, it holds for any $z$ on the closure of
    \begin{equation*}
    	\left\{\frac{1}{2}(1+i\theta): \theta \text{ a root of one of the forms $g_\Lambda$ or $h_M$ defined on \eqref{eq:forms}}\right\},
    \end{equation*}
    which is an uncountable set.
\end{theorem}

The $C$--expansions of the roots of the $C$--minimal forms are rather difficult to describe explicitly, see \cite{SchmidtCforms}. For this reason, we follow instead the approach developed by Harcos \cite{Harcos} in his unpublished undergraduate thesis, which works directly with binary quadratic forms. The previous results lead naturally to the following question. For $c>0$, define 
\begin{multline*}
    X_c = \bigg\{z\in\C\setminus\Q(i) : \abs{z-\frac{p}{q}}<\frac{1}{c|q|^2} \\ 
    \text{has only finitely many solutions } p,q\in\Z[i]\bigg\}.
\end{multline*}

The cardinality of $X_c$ is completely determined by Schmidt's theorem and our results: for $c<2$ this follows from \Cref{thm:schmidt}, for $c>2$ from \Cref{cor:dimension}, and for $c=2$ from \Cref{thm:uncountably_many_for_c=2}.

\begin{corollary} 
    The set $X_c$ is uncountable for $c\geq 2$ and countable for $c<2$. 
\end{corollary}

The paper is organized as follows. In \Cref{sec:preliminary} we recall some basic facts on continued fractions and prove \Cref{lem:A}. In \Cref{sec:binary_forms} we recall the theory of $C$--minimal forms and show how their roots can be used to construct uncountably many poorly approximable complex numbers. In \Cref{sec:dynamical_spectra} we review the theory of dynamically defined spectra. In \Cref{sec:dynamical_characterization}, we prove the dynamical characterization of the spectrum $L_{\frac{1}{2}+i\R}$ in terms of regular continued fractions. Finally, in \Cref{sec:fractal_properties}, we establish the fractal properties of the restricted spectrum $L_{\frac{1}{2}+i\R}$ and prove \Cref{thm:main}. We remark that \Cref{sec:binary_forms} is independent from the fractal geometry results and can be read independently.

\section{Basic facts of continued fractions}\label{sec:preliminary}

Given $\alpha \in \R\setminus\Q$, we define recursively $\alpha_0 := \alpha$,
\begin{equation*}
a_n = \lfloor \alpha_n \rfloor, \qquad \alpha_{n+1} = \frac{1}{\alpha_n-a_n}, \quad\text{for every } n\in\N.
\end{equation*}

Thus, for every $n\in\N$, we have $\alpha_n=a_n+\frac1{\alpha_{n+1}}$ and
\begin{equation*}
    \alpha = \alpha_0 = [a_0;a_1,a_2,\dots,a_n,\alpha_{n+1}]= a_0 + \cfrac1{a_1 + \cfrac1{a_2 +{\ddots +{\cfrac1{a_n+\cfrac1{\alpha_{n+1}}}}}}}.
\end{equation*}

The rational numbers
\begin{equation*}
    \frac{p_n}{q_n}=[a_0;a_1,a_2,\dots,a_n]=a_0 + \cfrac1{a_1 + \cfrac1{a_2 + {\ddots +{\cfrac1{a_n}}}}}, \qquad n \in \mathbb N,
\end{equation*}
are called the convergents of the continued fraction of $\alpha$ and
$$
\alpha = [a_0;a_1,a_2,\dots] = a_0 +\cfrac1{a_1 + \cfrac1{a_2 + {\ddots}}}.
$$
is the representation by continued fractions of the irrational real number $\alpha$. For every $n\in\mathbb N$, it holds that 
\begin{equation*}
    \left| \alpha - \frac{p_n}{q_n}\right| < \frac{1}{2q_n^2}
  \quad\text{or}\quad
  \left| \alpha - \frac{p_{n+1}}{q_{n+1}}\right| < \frac{1}{2q_{n+1}^2}.
\end{equation*}

On the other hand, if $\bigl| \alpha-\frac pq\bigr| < \frac{1}{2q^2}$ then $\frac pq$ is a convergent of the continued fraction of $\alpha$ by Legendre's theorem. We will need the following generalization \cite[Theorem 10, Page 16]{SergeLang}.

\begin{lemma}\label{lem:legendre_generalized}
    Let $\alpha=[a_0;a_1,\dots]$ and suppose $q_n\leq q<q_{n+1}$. If
    \begin{equation*}
        \abs{\alpha-\frac{p}{q}}\leq\frac{1}{q^2},
    \end{equation*}
    then one of the following holds for some $n\geq 0$:
    \begin{itemize}
        \item $\frac{p}{q}=\frac{p_n}{q_n}$,
        \item $a_{n+1}\geq 2$, $\frac{p}{q}=\frac{p_{n+1}-p_n}{q_{n+1}-q_n}$ and $a_{n+1}-2+\beta_{n+1}<\alpha_{n+2}$,
        \item $a_{n+1}\geq 2$, $\frac{p}{q}=\frac{p_{n}+p_{n-1}}{q_{n}+q_{n-1}}$ and $(\alpha_{n+1}-2)\beta_{n+1}<1$.
    \end{itemize}
\end{lemma}

We have the following exact formula for the error term:
\begin{equation*}
 \alpha-\frac{p_n}{q_n} =\frac{(-1)^n}{(\alpha_{n+1}+\beta_{n+1}) q_n^2},\quad \text{for all } n\in {\mathbb N},
\end{equation*}
where $p_n/q_n=[a_0;a_1,a_2,\dots,a_n]$, $\alpha_{n}=[a_n;a_{n+1},a_{n+2},\dots]$ and $\beta_n=[0;a_{n-1},a_{n-2},\dots,a_1]$. 

As consequence given $x=[a_0;a_1,\dots, a_n, a_{n+1},\dots]$ and \\ $\tilde{x}=[a_0;a_1,\dots,a_n,b_{n+1},\dots]$, we have
\begin{equation}\label{eq:compare_continued_fractions}
    x-\tilde{x} = (-1)^n\frac{\tilde{\alpha}_{n+1}-\alpha_{n+1}}{q_n^2(\beta_{n+1}+\alpha_{n+1})(\beta_{n+1}+\tilde{\alpha}_{n+1})} 
\end{equation}
where $\alpha_{n+1}=[a_{n+1};a_{n+2},\dots]$, $\tilde{\alpha}_{n+1}=[b_{n+1};b_{n+2},\dots]$ and $\beta_{n+1}=\frac{q_{n-1}}{q_n}=[0;a_n,\dots,a_1]$. 

\begin{lemma}
Given two continued fractions $[0;a_1,\dots,a_n,a_{n+1},\dots]$ and $[0;a_1,\dots,a_n,b_{n+1},\dots]$ whose coefficients are equal to each other until the $n$-th coefficient, we have the following inequality:
    $$\abs{[0;a_1,\dots,a_n,a_{n+1},\dots]-[0;a_1,\dots,a_n,b_{n+1},\dots]}<2^{-(n-1)}$$
\end{lemma}

Through the paper we will work mainly with quotients of Gaussian integers with real part $1/2$. Given such an approximation $\frac{p}{q}=\frac{1}{2}+i\frac{r}{s}$ for some $r/s\in\Q$ and $\gcd(p,q)=\gcd(r,s)=1$, then the reduced form of $\frac{p}{q}$ is equal to 
\begin{equation}\label{eq:denominator_cases}
    \frac{p}{q}=\frac{1}{2}+i\frac{r}{s}=
    \begin{cases}
        \frac{s+2ri}{2s}, & \text{if $s\equiv1\pmod{2}$}, \\
        \frac{s/2+ri}{s}, & \text{if $s\equiv0\pmod{4}$}, \\
        \frac{(s/2+r)/2+i(r-s/2)/2}{s/(i+1)}, & \text{if $s\equiv2\pmod{4}$}.
    \end{cases}
\end{equation}

Indeed, first recall that if $a,b\in\Z$ are relatively prime over $\Z$, then there are $u,v\in\Z$ such that $au+bv=1$, so they are also relatively prime over $\Z[i]$. If $s$ is odd, then $\gcd(s,2r)=1$  and $\gcd(2,s+2ri)=\gcd(2,s)=1$. If $s\equiv 0\pmod{4}$, then $\gcd(2,s/2+ri)=\gcd(2,ri)=1$ and $\gcd(s/4,s/2+ri)=\gcd(s/4,ri)=1$. If $s\equiv 2\pmod{4}$, then we use the fact that $\gcd(s/2,s/2+ri)=\gcd(s/2,ri)=1$ and $\gcd(2,s/2+ri)=1+i$.

We will now characterize the good approximations and approximation constants of numbers of the form $\frac{1}{2}+i\alpha$ where $\alpha\in \R\setminus \Q$. We will see later in \Cref{rem:optimal_inequality} that the constant $\frac{1}{2}$ in the following lemma is optimal.

\begin{lemma}\label{lem:fundamental_lemma}
Let $p,q\in \Z[i]$ and $\alpha\in \R\setminus\Q$. If 
\begin{equation*}
    \abs{\frac{1}{2}+i\alpha-\frac{p}{q}}\leq\frac{1}{2\abs{q}^2},
\end{equation*}
then $\Re(p/q)=1/2$. Moreover, for any such approximation with $\Re(p/q)=1/2$, if $(p_m/q_m)_{m\geq 0}$ are the convergents of the continued fraction of $\alpha$,  then one of the following holds for some $n\geq 0$:
\begin{itemize}
    \item $\Im(p/q)=\frac{p_{n+1}-p_n}{q_{n+1}-q_n}$ and $q_{n+1}-q_n\equiv 2\pmod{4}$; or
    \item $\Im(p/q)=\frac{p_{n}+p_{n-1}}{q_{n}+q_{n-1}}$ and $q_n+q_{n-1}\equiv 2\pmod{4}$; or
    \item $\Im(p/q)=\frac{p_n}{q_n}$.
\end{itemize}

\end{lemma}

\begin{proof}
Write $p=a+bi$, $q=c+di$. Then 

\begin{equation*}
    \Re\left(\frac{a+bi}{c+di}\right)=\frac{1}{2}\left(\frac{a+bi}{c+di}+\frac{a-bi}{c-di}\right)=\frac{ac+bd}{c^2+d^2}.
\end{equation*}

We see that $\frac{ac+bd}{c^2+d^2}$ must be equal to $\frac{1}{2}$. In fact, if this quotient is not equal to $1/2$, since $\alpha$ is irrational we get
\begin{equation*}
    \frac{1}{2(c^2+d^2)}\geq\abs{\frac{1}{2}+i\alpha-\frac{p}{q}}>\abs{\frac{ac+bd}{c^2+d^2}-\frac{1}{2}}\geq\frac{1}{2(c^2+d^2)}.
\end{equation*}
In particular, we have that $\frac{p}{q}=\frac{1}{2}+i\frac{r}{s}$ for some $r/s\in\Q$ with $\gcd(r,s)=1$. We will now divide the proof into the cases that appear in \eqref{eq:denominator_cases}. 

If $s$ is odd or $s\equiv 0\pmod{4}$, then $q=sk$ for some $k\in\Z[i]$. In particular, the inequality
\begin{equation*}
    \abs{\frac{1}{2}+i\alpha-\frac{p}{q}}=\abs{\alpha-\frac{r}{s}}\leq\frac{1}{2s^2|k|^2}
\end{equation*}
implies that $r/s$ is a convergent of $\alpha$ by Legendre's theorem. 

If $s\equiv 2\pmod{4}$, then $q=sk/(1+i)$ for some $k\in\Z[i]$. In particular,
\begin{equation*}
    \abs{\frac{1}{2}+i\alpha-\frac{p}{q}}=\abs{\alpha-\frac{r}{s}}\leq\frac{1}{s^2|k|^2}.
\end{equation*}
If $|k|>1$, then $r/s$ is again a convergent of $\alpha$. However, if $|k|=1$, then we have $\abs{\alpha-\frac{r}{s}}\leq\frac{1}{s^2}$, so by \Cref{lem:legendre_generalized} we have that either $\frac{r}{s}=\frac{p_n}{q_n}$, $\frac{r}{s}=\frac{p_{n+1}-p_n}{q_{n+1}-q_n}$, or $\frac{r}{s}=\frac{p_n+p_{n-1}}{q_n+q_{n-1}}$. 
\end{proof}

\section{Binary quadratic forms}\label{sec:binary_forms}

We begin with binary quadratic forms with real coefficients. Given two such forms $f(x,y)$ and $g(x,y)$, we say that they are \emph{$GL_2(\Z)$-equivalent} if there exists $A=\left(\begin{smallmatrix} a & b \\ c & d \end{smallmatrix}\right)\in GL_2(\Z)$ such that $f(ax+by,cx+dy)=g(x,y)$. Recall that a real quadratic form $f(x,y)=Ax^2+Bxy+Cy^2$ is indefinite, that is, takes both positive and negative values, if and only if its discriminant $B^2-4AC$ is positive.

Recall that a Markov number is the largest coordinate of a positive integer solution triple $(x,y,z)$ of the Markov equation $x^2+y^2+z^2=3xyz$. The \emph{multiplicity} of a Markov number $m$ is the number of distinct solutions $(x,y,z)$ with $x\leq y\leq z$ of Markov equation with $z=m$. Let $(z_r)_{r\geq1}$ be the sequence of the ordered Markov numbers with multiplicity, i.e., a non-decreasing sequence whose terms are all Markov numbers and such that the number of times each term appears is equal to its multiplicity. 

Suppose that $(x_r,y_r,z_r)$ is a positive integer solution of Markov's equation $x^2+y^2+z^2=3xyz$ with $\max\{x_r, y_r\}\leq z_r$. The solution is called singular if $x_r=y_r$. It is easy to see that the only singular solutions are $(1,1,1)$ and $(1,1,2)$.

Let $0\leq k_r<z_r$ and $l_r>0$ be the integers defined by
\begin{equation}\label{eq:k-l}
    k_r\equiv\frac{y_r}{x_r}\equiv-\frac{x_r}{y_r}\pmod{z_r} \qquad k_r^2+1=l_rz_r.
\end{equation}
We order $x_r$ and $y_r$ such that $k_r$ is minimal. In particular $0\leq 2k_r\leq z_r$.

\begin{definition}
    Let $z_r$ be a Markov number. Then the \emph{Markov form} $F_{r}$ is the binary quadratic form 
    \begin{equation}\label{eq:forms_F_r}
        z_rF_{r}=z_rx^2+(3z_r-2k_r)xy+(l_r-3k_r)y^2
    \end{equation}
    where $k_r$ and $l_r$ are defined on \eqref{eq:k-l}.
\end{definition}

The definition of $k_r$ and $l_r$ is asymmetric on $x_r,y_r$. If we exchange $x_r$ with $y_r$, the corresponding $k_r^\prime, l_r^\prime$ defined on \eqref{eq:k-l} yield the equivalent form $F_{r}^\prime(x,y)=F_{r}(x+2y,-y)$. 

In spite of the fact that it is still unknown whether there can exist different positive solutions $(x_r,y_r,z_r)$ and $(x_s,y_s,z_r)$ of Markov's equation with $\{x_r,y_r\}\neq\{x_s,y_s\}$, the triple determines uniquely $F_r$, up to permuting $(x_r,y_r)$. In fact, $F_r$ and $F_s$ are nonequivalent for $r\neq s$, because their roots are $GL_2(\Z)$-nonequivalent between each other because of \cite[Theorem 18]{Bombieri}. On the other hand, it is known \cite[Chapter II, Lemma 10]{Cassels} that $\abs{F_r(x,y)}\geq 1$ for all $(x,y)\in\Z^2\setminus\{(0,0)\}$.

We now turn to complex binary quadratic forms. Let $f(x,y)=Ax^2+Bxy+Cy^2$ with $A,B,C\in\C$ be such that
\begin{equation*}
    A\neq 0, \quad B^2-4AC\neq 0.
\end{equation*}

Recall that $\mathcal{N}(\Lambda)=\{1, 5, 29, 65, 169, 349,\dots\}$ is the set of positive integer solutions $x$ of Vulakh's \cite{Vulakh} equation $2x^{2}+y_{1}^{2}+y_{2}^{2}=4y_{1}y_{2}x$, which is equivalently written as
\begin{equation}\label{eq:Vulakh}
\left\{
\begin{aligned}
    x_1+x_2 &= 2y_1y_2,\\
    2x_1x_2 &= y_1^2+y_2^2.
\end{aligned}\right.
\end{equation}

\begin{figure}
\centering
\begin{forest}
for tree={
    draw=none,
    edge path={
        \noexpand\path[\forestoption{edge}]
        (!u.parent anchor) -- +(0,-30pt)
        -| (.child anchor)\forestoption{edge label};
    },
    l sep=12mm,
    s sep=12mm,
}   
[{$(1,1;1,1)$}
    [{$(1,5;1,3)$}
        [{$(1,65;3,11)$}
            []
            [$\vdots$]
            []
        ]
        [{$(5,349;3,59)$}
            []
            [$\vdots$]
            []
        ]
        [{$(5,29;1,17)$}
            []
            [$\vdots$]
            []
        ]
    ]
]
\end{forest}
\caption{Solutions of Vulakh's equation \eqref{eq:Vulakh}.}
\label{fig:Vulakh_tree}
\end{figure}
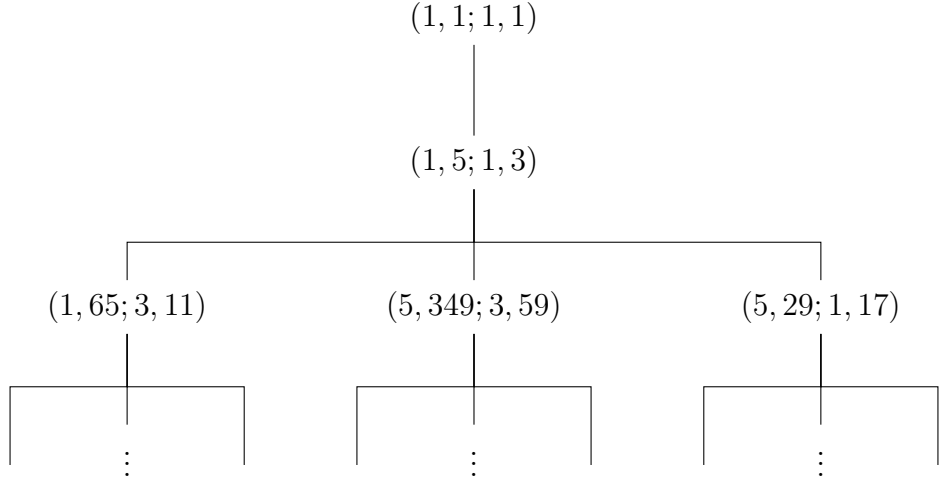

We will now state several facts about the solutions of Vulakh's equation, see \cite[Section 5]{Sch}. A solution $(\Lambda_1,\Lambda_2;M_1,M_2)$ of \eqref{eq:Vulakh} is called singular if $\Lambda_1=\Lambda_2$ or $M_1=M_2$. Two solutions $(\Lambda_1,\Lambda_2;M_1,M_2)$ and $(\Lambda_1^\ast,\Lambda_2^\ast;M_1^\ast,M_2^\ast)$ are neighbors if $\{\Lambda_1,\Lambda_2\}\cap\{\Lambda_1^\ast,\Lambda_2^\ast\}\neq\varnothing$ and  $\{M_1,M_2\}\cap\{M_1^\ast,M_2^\ast\}\neq\varnothing$. There is precisely one singular solution, namely $(1,1;1,1)$ and this solution has the unique neighbor $(1,5;1,3)$. Moreover, for any solution we have $\gcd(\Lambda_1,\Lambda_2)=\gcd(M_1,M_2)=1$. 

Given a solution $(\Lambda_1,\Lambda_2;M_1,M_2)$ of \eqref{eq:Vulakh} with $\Lambda_1\leq\Lambda_2$ and $M_1\leq M_2$, it is easy to see that
\begin{equation}\label{eq:M_2>Lambda_1}
    \Lambda_2\geq M_1M_2, \quad M_2\geq M_1\Lambda_1,
\end{equation}
with strict inequality when the solution is not singular.

Every non singular solution $(\Lambda_1,\Lambda_2;M_1,M_2)$ with $\Lambda_1<\Lambda_2$ and $M_1<M_2$ has precisely four different neighbors solutions 
\begin{equation}\label{eq:neighbors_solutions}
    (\Lambda_i,\Lambda_{ij};M_j,M_{ij}), \quad i,j\in\{1,2\}, 
\end{equation}
where
\begin{equation*}
	\Lambda_{ij}=\Lambda_i(8M_j^2-1)-(\Lambda_1+\Lambda_2), \quad M_{ij}=4\Lambda_iM_j-M_k,
\end{equation*}
with $k\in\{1,2\}\setminus\{j\}$.

For each nonsingular solution we associate a quintuple $(\varepsilon,l_1,l_2;m_1,m_2)$ as follows: for $(1,5;1,3)$ we associate $(1,0,2;1,1)$ and given
\begin{equation*}
    (\Lambda_1,\Lambda_2;M_1,M_2)\mapsto(\varepsilon,l_1,l_2;m_1,m_2),
\end{equation*}
then for $(i,j)=(1,2),(2,1),(2,2)$ we associate
\begin{equation*}
	(\Lambda_i,\Lambda_{ij};M_j,M_{ij})\mapsto(\varepsilon_{ij},l_i,l_{ij};m_j,m_{ij}), 
\end{equation*}
where
\begin{align*}
    \varepsilon_{12}&=\varepsilon_{21}=\varepsilon, \varepsilon_{22}=-\varepsilon, \\
    l_{ij}&=(\Lambda_{ij}m_j-\varepsilon_{ij}M_{ij})/M_j, \\
    m_{ij}&=(M_{ij}l_i+\varepsilon_{ij}M_j)/\Lambda_i.
\end{align*}

Given a solution $(\Lambda_1,\Lambda_2;M_1,M_2)$, $\Lambda_1<\Lambda_2$, $M_1<M_2$ with respective quintuple $(\varepsilon,l_1,l_2;m_1,m_2)$, it can be seen by induction that it holds
\begin{equation}\label{eq:determinant_equation}
    \Lambda_1l_2-\Lambda_2l_1=2\varepsilon M_1^2 \quad M_1m_2-M_2m_1=-2\varepsilon\Lambda_1,
\end{equation}
\begin{equation}\label{eq:determinant_equation2}
    \Lambda_2m_2-M_2l_2=-\varepsilon M_1,
\end{equation}
and that for $i=1,2$ it holds
\begin{equation}\label{eq:congruences}
	l_i \equiv 0 \pmod{2}, \qquad m_i\equiv 1\pmod{2},
\end{equation}
\begin{equation}\label{eq:l_less_than_Lambda}
	0\leq l_i\leq \Lambda_i, \qquad 0\leq m_i\leq M_i,
\end{equation}	
and
\begin{equation}\label{eq:congruences2}
	l_i^2+1\equiv 0\pmod{\Lambda_i}, \qquad m_i^2+2\equiv 0 \pmod{M_i},
\end{equation}

In particular we can denote for $i=1,2$, the positive integers
\begin{equation}\label{eq:lambda_i_mu_i}
    \lambda_i:=\frac{l_i^2+1}{\Lambda_i}, \qquad \mu_i:=\frac{m_i^2+2}{M_i}.
\end{equation}

Now we define the families of binary quadratic rational forms $g_\Lambda$ and $h_M$ by
\begin{equation}\label{eq:forms}
\begin{aligned}
    \Lambda g_\Lambda(x,y)
        &= \Lambda x^2+(4\Lambda-2l)xy+(\lambda-4l)y^2, \\
    Mh_M(x,y)
        &= Mx^2+(4M-2m)xy+(\mu-4m)y^2.
\end{aligned}
\end{equation}

The $C$--expansions of the roots of these families of forms were studied by Schmidt in \cite{SchmidtCforms}. It is likely that his arguments also imply that these roots are pairwise nonequivalent. In any case, Schmidt proved that the forms $g_\Lambda$ and $h_M$ are in one-to-one correspondence with a certain family of Markov symbols. As a consequence, all these forms are pairwise distinct. 

On the other hand, it is unknown to the authors whether the sequence $\mathcal{N}(\Lambda)$ satisfies an asymptotic formula analogous to the one obtained by Zagier \cite{Zagier} for the classical Markov numbers. Likewise, it is unclear whether there is a natural analogue of the Markov uniqueness conjecture in this setting.

It can be shown (see \cite[Lemma 5.7]{Sch}) that 
\begin{equation}\label{eq:special_values_of_g_and_h}
\begin{aligned}
    g_{\Lambda_2}(l_2,\Lambda_2)&=g_{\Lambda_2}(l_2-4\Lambda_2,\Lambda_2)=1, \\
    h_{M_2}(m_2,M_2)&=h_{M_2}(m_2-4M_2,M_2)=2.
\end{aligned}
\end{equation}

Define the binary real forms 
\begin{equation*}
    g_\Lambda^\prime(y,z):=\Lambda^2 g_\Lambda(x,y), \qquad h_M^\prime(y,w)=M^2h_M(x,y)
\end{equation*}
where $z=\Lambda x- ly$ and $w=Mx-my$. So
\begin{equation*}
    g_\Lambda^\prime(y,z)=y^2+4\Lambda yz+z^2, \qquad h_M^\prime(y,w)=2y^2+4M yw+w^2. 
\end{equation*}
Denote
\begin{equation*}
    \Z_j^2=\{(x,y)\in\Z^2\setminus\{(0,0)\}:\gcd(x,y)=1,x+y\equiv j\pmod{2}\}.
\end{equation*}
We have \cite[Lemma 5.9]{Sch}
\begin{lemma}\label{lem:schmidt_59}
For any of the forms $g_\Lambda$ and $h_M$ defined in \eqref{eq:forms} we have
\begin{equation*}
    \abs{g_\Lambda(x,y)}\geq 2 \quad\text{and}\quad \abs{h_M(x,y)}\geq 2, \qquad \text{for all }(x,y)\in\Z_0^2,
\end{equation*}
\begin{equation*}
    \abs{g_\Lambda(x,y)}\geq 1 \quad\text{and}\quad \abs{h_M(x,y)}\geq 1, \qquad \text{for all }(x,y)\in\Z_1^2.
\end{equation*}
\end{lemma}

We now adapt the method of \cite{Harcos}. Since the arguments apply uniformly to the families of forms $g_\Lambda$ and $h_M$, as well as to the forms $F_r$ defined in \eqref{eq:forms_F_r}, it is convenient to formulate them in a more general setting.

Let $m,l,n,t$ be positive integers with $n\geq 2$ and $k$ an integer such that $m\geq k\geq 1$, $n>t+1/m$ and $k^2+t=lm$. Assume moreover that $n^2m^2-4t$ is not a perfect square. Define the binary quadratic form $F(x,y)$ by 
\begin{equation}\label{eq:form_F}
    mF(x,y):=mx^2+(nm-2k)xy+(l-nk)y^2.
\end{equation}

Let $F(x,y)=(x-\theta y)(x-\Theta y)$ with $\Theta<\theta$ and set
\begin{equation*}
    \Delta := \frac{n+\sqrt{n^2-4tm^{-2}}}{2}.
\end{equation*}
Note that $\theta\in\R\setminus\Q$ and $\Theta\in\R\setminus\Q$.

\begin{lemma}\label{lem:roots_of_F}
    The roots $\Theta<\theta$ of the quadratic form \eqref{eq:form_F} satisfy
    \begin{equation}\label{eq:theta_and_Theta}
        \theta = \frac{k}{m}-n+\Delta, \qquad \Theta = \frac{k}{m}-\Delta.
    \end{equation}
    Moreover $-n<\Theta<-(n-1)<0<\theta<1$.
\end{lemma}

\begin{proof}
    From \eqref{eq:form_F}, we can write
    \begin{align*}
        F(x,y)&=x^2+\left(n-\frac{2k}{m}\right)xy+\left(\frac{l-nk}{m}\right)y^2 \\
        &=\left(x+\left(\frac{n}{2}-\frac{k}{m}\right)y\right)^2+\left(\left(\frac{l-nk}{m}\right)-\left(\frac{n}{2}-\frac{k}{m}\right)^2\right)y^2 \\
        &=\left(x+\left(\frac{n}{2}-\frac{k}{m}\right)y\right)^2-\left(\frac{n^2}{4}-\frac{t}{m^2}\right)y^2. 
    \end{align*}
    Hence we obtain \eqref{eq:theta_and_Theta}.

    It is easy to see that $\theta<k/m\leq 1$ and since $\Delta<n$, it follows $\Theta\geq-\Delta>-n$. If $m=k$, then $\theta>0$ is equivalent to $\sqrt{n^2-4tm^{-2}}>n-2$, which is true whenever $n>1+tm^{-2}$, which holds by hypothesis because $t+m^{-1}\geq 1+tm^{-2}$. Now assume that $m\neq k$, so $k\leq m-1$. Hence $\Theta\leq 1-\frac{1}{m}-\Delta$. Notice that $\Theta<-(n-1)$ will then follow from
    \begin{equation*}
        n-\frac{2}{m}<\sqrt{n^2-\frac{4t}{m^2}},
    \end{equation*}
    which is equivalent to $\frac{t+1}{m}< n$. Finally, notice that $\theta>0$ is equivalent
    \begin{equation*}
        \frac{n^2}{4}-\frac{t}{m^2} \geq\left(\frac{n}{2}-\frac{k}{m}\right)^2=\frac{n^2}{4}-\frac{kn}{m}+\frac{lm-t}{m^2},
    \end{equation*}
    which in turn is equivalent to $kn\geq l$, which is equivalent to $nm\geq lm/k=k+t/k$ which is true because $nm> tm+1\geq tk+1\geq k+t/k$.
\end{proof}

    Now let us define the normalized form
    \begin{equation}\label{eq:normalization_F}
        F^\prime(y,z):=m^2F(x,y)=ty^2+nmyz+z^2
    \end{equation}
    where $z=mx-ky$. Note that we have
    \begin{equation}\label{eq:identity_F}
        F^\prime(y,z)=F^\prime(-y,z+nmy).
    \end{equation}
    
    For $s\in\mathbb Z\setminus\{0\}$, define
    \begin{equation*}
    	c_4(s):=\begin{cases}
    		1, & \text{if } s\equiv 1\pmod{2}, \\
    		4, & \text{if } s\equiv 0\pmod{4}, \\
    		8, & \text{if } s\equiv 2\pmod{4}.
    	\end{cases}
    \end{equation*}
    This quantity is defined according to \eqref{eq:denominator_cases}, because if $\frac{p}{q}=\frac{1}{2}+i\frac{r}{s}$ with $\gcd(p,q)=\gcd(r,s)=1$, then 
    \begin{equation}\label{eq:c_4(s)}
    	c_4(s)=\frac{4s^2}{\abs{q}^2}.
    \end{equation}

	Now we specialize to the case we are interested, which is $n=4$, $t\in\{1,2\}$ and $k,m$ satisfying \eqref{eq:congruences}. 

\begin{lemma}\label{lem:theta}
    Let $F=g_\Lambda$ or $F=h_M$ be one of the forms defined in \eqref{eq:forms} and let $\theta>0$ denote the positive root of $F$, that is, $F(\theta,1)=0$. Then, for every $r,s\in\Z$ with $\gcd(r,s)=1$ and $s\neq 0$,
    \begin{equation*}
        \Delta s^2\abs{\theta - \frac{2r}{s}}\geq c_4(s),
    \end{equation*}
    with equality if and only if $(2r,s)=(k,m)$. 
\end{lemma}

\begin{proof}
    First observe that if $2r/s\not\in[0,1]$, then, since $\theta\in(0,1)$, we will have that $s^2\abs{\theta - \frac{2r}{s}}\geq 2|s|\geq c_4(s)$. Since $\Delta\geq 2$, the desired inequality follows immediately. Therefore, we may assume that $2r/s\in[0,1]$, and hence, after replacing $(r,s)$ by $(-r,-s)$ if necessary, that $r,s\geq 0$. 
 
 	By \Cref{lem:schmidt_59}, we have $\abs{F(2x,y)}\geq c_4(y)$ for all $x,y\in\Z$ with $\gcd(x,y)=1$. Indeed, if $y$ is odd, then $(2x,y)\in\Z_1^2$, and the claim follows from the first inequality in \Cref{lem:schmidt_59}. If $y\equiv 0\pmod 4$, then $\abs{F(2x,y)}=4\abs{F(x,y/2)}\geq 4 =  c_4(y)$. Finally, if $y\equiv 2\pmod 4$, then $x$ is odd because $\gcd(x,y)=1$, so $(x,y/2)\in\Z_0^2$, and therefore $\abs{F(2x,y)}=4\abs{F(x,y/2)}\geq 8=c_4(y)$.
        
    The desired inequality is equivalent to
    \begin{equation*}
        \frac{\Delta}{c_4(s)}\abs{F(2r,s)}\geq \abs{\Theta-\frac{2r}{s}}.
    \end{equation*}
    By \Cref{lem:roots_of_F}, and since $2r/s\geq 0>\Theta$, the previous inequality is equivalent to \begin{equation*}
        \Delta\left(\frac{\abs{F(2r,s)}}{c_4(s)}-1\right)\geq \frac{2r}{s}-\frac{k}{m}.
    \end{equation*}
    Set 
    \begin{equation*} 
        u:=2mr-ks. 
    \end{equation*} 
    Multiplying both sides by $m^2$ and using \eqref{eq:normalization_F}, we see that it suffices to prove \begin{equation}\label{eq:inequality_to_be_proven1}
        \Delta\left(\frac{\abs{F^\prime(s,u)}}{c_4(s)}-m^2\right)\geq \frac{mu}{s}.
    \end{equation}
    Moreover, \eqref{eq:normalization_F} gives
    \begin{equation}\label{eq:inequality_c_4_u}
    	\abs{F^\prime(s,u)}=m^2\abs{F(2r,s)}\geq m^2c_4(s).	
    \end{equation}

    If $u=0$, then equality holds if and only if $(2r,s)=(k\sqrt{c_4(s)},m\sqrt{c_4(s)})$. Since $\sqrt{c_4(s)}$ is an integer only when $c_4(s)\in\{1,4\}$ and $2m\equiv 2\pmod{4}$, the tuple $(k\sqrt{c_4(s)},m\sqrt{c_4(s)})$ has integral coordinates if and only if $c_4(s)=1$. Consequently, equality can occur only when $c_4(s)=1$. Since the left-hand side of \eqref{eq:inequality_to_be_proven1} is non-negative, we may therefore assume that $u>0$, and we shall prove that the inequality in \eqref{eq:inequality_to_be_proven1} is strict. 
    
    By \eqref{eq:congruences} and \eqref{eq:congruences2}, the integer $m$ is odd. In particular, $\gcd(m,t)=1$. Since $u\equiv - ks\pmod{m}$ and $k^2+t=lm$, we obtain $ku\equiv ts\pmod{m}$. Set $s^\prime:=u$ and $u^\prime:=-ts$. Then $u^\prime\equiv -ks^\prime\pmod{m}$ and hence there exists an integer $r'$ such that $u^\prime=mr^\prime-ks^\prime$. Consequently,
    \begin{equation}\label{eq:F_s_u_prime}
    	F^\prime(s^\prime,u^\prime)=t(s^\prime)^2+nms^\prime u^\prime+(u^\prime)^2=t(u^2- nmus+ts^2).
    \end{equation}
    On the other hand, \eqref{eq:normalization_F} yields $\abs{F^\prime(s^\prime,u^\prime)}=m^2\abs{F(r^\prime,s^\prime)}$. Since $\gcd(m,t)=1$, it follows from \eqref{eq:F_s_u_prime} that $t$ divides $F(r',s')$. We claim that either
    \begin{equation}\label{eq:case_r_s_prime}
    	\abs{F(r^\prime,s^\prime)}\geq \max\{t,c_4(s)\} \quad\text{and}\quad \min\{t,c_4(s)\}=1,
    \end{equation} 
    or 
    \begin{equation}\label{eq:case_r_s_prime2}
    	\abs{F(r^\prime,s^\prime)}\geq 2c_4(s).
    \end{equation}

    If $s$ is odd, then $c_4(s)=1$, and therefore \eqref{eq:case_r_s_prime} follows immediately. Assume now that $s$ is even. Then $u$ is also even, and since $m$ is odd, the relation $-ts=u^\prime=mr'-ks'=mr^\prime-ku$ shows that $r'$ is even. Hence, by \eqref{eq:inequality_c_4_u}, 
    \begin{equation*}
    	\abs{F(r^\prime,s^\prime)}\geq \max\{t,c_4(u)\}.
    \end{equation*} 
    Moreover, since $\gcd(r,s)=1$ and $s$ is even, we must have $r$ odd.
    
    If $s \equiv 0\pmod{4}$, then $u=2mr-ks\equiv 2\pmod{m}$ and therefore $c_4(u)=8=2c_4(s)$. Thus \eqref{eq:case_r_s_prime2} holds. Finally, assume $s\equiv 2\pmod{4}$. If $t=1$, then \eqref{eq:congruences} implies that $k$ is even. Consequently, $u\equiv 2mr\pmod{4} \equiv 2\pmod{4}$ and hence $c_4(u)=8=c_4(s)$. Therefore, \eqref{eq:case_r_s_prime} holds. If $t=2$, then $k$ is odd, and so $u\equiv  2mr-2\pmod{4} \equiv 0 \pmod{4}$. Consequently, $r^\prime=(ku-ts)/m$ is divisible by $4$. Therefore $\abs{F(r^\prime,s^\prime)}=16\abs{F(r^\prime/4,u/4)}\geq 16 = 2c_4(s)$, and hence \eqref{eq:case_r_s_prime2} follows.
    
    Set $c=c_4(s)$. Assume first that $F^\prime(s^\prime,u^\prime)\geq 0$. If \eqref{eq:case_r_s_prime} holds, then using \eqref{eq:F_s_u_prime} we obtain
    \begin{align*}
    	\frac{ts^2+nmsu+u^2}{c}-m^2&\geq \frac{2nmsu+m^2\max\{t,c\}/t}{c}-m^2  \\
    	&\geq \frac{2nmsu}{c}+m^2(1/\min\{t,c\}-1).
    \end{align*}
    Since $\min\{t,c\}=1$, we obtain 
    \begin{equation*} 
        \frac{ts^2+nmsu+u^2}{c}-m^2 \geq \frac{nmsu}{c}. 
    \end{equation*} 
    The same inequality clearly remains true when \eqref{eq:case_r_s_prime2} holds. Consequently, \begin{equation*} 
        \frac{\abs{F'(s,u)}}{c}-m^2 = \frac{ts^2+nmsu+u^2}{c}-m^2 \geq \frac{nmsu}{c} > \frac{mu}{s}, 
    \end{equation*} 
    since $n\geq 4$.

    Now suppose that $F'(s',u')<0$. If \eqref{eq:case_r_s_prime} holds, then $F(s^\prime,u^\prime)<-m^2\max\{t,c\}$. Using \eqref{eq:inequality_c_4_u} and \eqref{eq:F_s_u_prime}, we obtain 
    \begin{equation*}
    	tm^2c-2tnmus\leq tF^\prime(s,u) - 2tnmus = F^\prime(s^\prime,u^\prime) < -m^2\max\{t,c\}.
    \end{equation*}
    Hence $(tc+\max\{t,c\})m^2<2tnmus$. Therefore,
    \begin{align*}
    	\frac{ts^2+nmsu+u^2}{c}-m^2&\geq \frac{ts^2+(m^2/2)(tc+\max\{t,c\})/t+u^2}{c}-m^2  \\
    	&\geq\frac{ts^2+u^2}{c}+\left(\frac{m^2}{2}\right)\left(\frac{1}{\min\{t,c\}}-1\right). 
    \end{align*}
    
    Finally, if \eqref{eq:case_r_s_prime2} holds, the same argument yields 
    \begin{equation*} 
        \frac{\abs{F'(s,u)}}{c}-m^2 = \frac{ts^2+nmsu+u^2}{c}-m^2 \geq \frac{ts^2+u^2}{c}. 
    \end{equation*}

    Since $mc<nus$ in this case, the right-hand side of \eqref{eq:inequality_to_be_proven1} is bounded above by $nu^2/c$. Therefore, it suffices to prove that 
    \begin{equation*}
        \Delta s^2\geq (n-\Delta)u^2 = \frac{u^2}{\Delta m^2}.
    \end{equation*}
    Equivalently, it suffices to show that $\Delta ms > u$. But this is immediate, since $u=2mr-ks\leq 2mr\leq ms$, where the last inequality follows from $2r/s\in[0,1]$. As $\Delta> 1$, we conclude that $\Delta ms> ms\geq u$. This completes the proof. 
    \end{proof}

\begin{lemma}\label{lem:Theta}
    Let $F=g_\Lambda$ or $F=h_M$ be one of the forms defined in \eqref{eq:forms} and let $\Theta<0$ denote the negative root of $F$, that is, $F(\Theta,1)=0$. Then, for every $r,s\in\Z$ with $\gcd(r,s)=1$ and $s\neq 0$,
    \begin{equation*}
    	\Delta s^2\abs{\Theta - \frac{2r}{s}}\geq c_4(s),
    \end{equation*}
    with equality if and only if $(2r,s)=(k,m)$. 
\end{lemma}

\begin{proof}
    The proof is analogous to that of \Cref{lem:theta}. As in the previous lemma, we may assume that $2r/s\in[-n,-n+1]$ and that $\abs{F(2r,s)}\geq c_4(s)$. Repeating the computations from the proof of \Cref{lem:theta}, we are reduced to proving
    \begin{equation*}
        \Delta\left(\frac{\abs{F^\prime(s,v)}}{c_4(s)}-m^2\right)\geq\frac{mv}{s},
    \end{equation*}
    where 
    \begin{equation*}
        v:=ks-2mr-nms=-u-nms
    \end{equation*}
    Moreover, using \eqref{eq:identity_F}, we obtain
    \begin{equation*}
        F^\prime(s,u)=F^\prime(-s,-u)=F^\prime(-s,v+nms)=F^\prime(s,v).
    \end{equation*}
    Hence the desired inequality is
    \begin{equation}\label{eq:inequality_to_be_proven2}
        \Delta\left(\frac{\abs{F^\prime(s,v)}}{c_4(s)}-m^2\right)\geq\frac{mv}{s}.
    \end{equation}
    Set $c:=c_4(s)$. Proceeding exactly as in the proof of \Cref{lem:theta}, we may assume that $v>0$ and $F^\prime(s,v)\geq m^2c$. Then
    \begin{equation*}
        \frac{F^\prime(s,v)}{c}-m^2=\frac{ts^2+nmsv+v^2}{c}-m^2.
    \end{equation*}
    
    Define the even integer $r^\prime:=-2r-ns$. Then $v=mr^\prime+ks$, and \eqref{eq:normalization_F} yields $\abs{F^\prime(-s,v)}=m^2\abs{F(r^\prime,-s)}\geq m^2c$ and
    \begin{equation*}
        F^\prime(-s,v)=ts^2-nmvs+v^2.
    \end{equation*}
    Suppose first that $F^\prime(-s,v)\geq 0$. Then $F^\prime(-s,v)\geq m^2c$, and therefore
    \begin{equation*}
        \frac{ts^2+nmsv+v^2}{c}-m^2\geq \frac{2nmsv+m^2c}{c}-m^2=\frac{2nmsv}{c} > \frac{mv}{s}.
    \end{equation*}
    Now assume that $F^\prime(-s,v)<0$. Then $F^\prime(-s,v)<-m^2c$. Since $v=-u-nms$, the identity \eqref{eq:identity_F} gives $F^\prime(s,v)=F^\prime(-s,v+nms)=F^\prime(-s,-u)=F^\prime(s,u)$. Moreover, since $u = 2mr-ks$ and $2r$ is even, \eqref{eq:normalization_F} implies that $\abs{F^\prime(s,u)}=m^2\abs{F(2r,s)}\geq m^2c$. Since $s>0$ and $v>0$, we obtain
    \begin{equation*}
        m^2c-2nmvs\leq F^\prime(s,v) - 2nmvs = F^\prime(-s,v) < -m^2c,
    \end{equation*}
    and hence $mc<nvs$. Therefore,
    \begin{equation*}
        \frac{ts^2+nmsv+v^2}{c}-m^2\geq\frac{ts^2+v^2}{c}.
    \end{equation*}

    As in the proof of \Cref{lem:theta}, the strict inequality in \eqref{eq:inequality_to_be_proven2} is equivalent to $\Delta ms>v$. Since $k\leq m$ and $2r/s\geq -n$, we have $v=ks-2mr-nms\leq ms$, and the conclusion follows because $\Delta>1$. 
\end{proof}

\begin{theorem}\label{thm:roots_inequality_theorem}
Let $F=g_\Lambda$ or $F=h_M$ be any of the forms defined on \eqref{eq:forms}. Let $\eta$ be a root of $F$, that is, $F(\eta,1)=0$. Then, for all $p,q\in\Z[i]$ with $q\neq 0$, we have
\begin{equation*}
    \abs{\frac{1+i\eta}{2} - \frac{p}{q}}>\frac{1}{2\abs{q}^2}.
\end{equation*}
\end{theorem}

\begin{proof}
    Assume, for contradiction, that there exist $p,q\in\Z[i]$ with $q\neq 0$ and $\gcd(p,q)=1$ such that 
    \begin{equation}\label{eq:contradiction_inequality}
        \abs{\frac{1+i\eta}{2} - \frac{p}{q}}\leq\frac{1}{2\abs{q}^2}
    \end{equation}
    By \Cref{lem:fundamental_lemma}, we have $\Re(p/q)=1/2$, and hence $\frac{p}{q}=\frac{1}{2}+i\frac{r}{s}$ for some $r/s\in\Q$ with $\gcd(r,s)=1$. Therefore, using \eqref{eq:c_4(s)}, the inequality \eqref{eq:contradiction_inequality} becomes
    \begin{equation*}
        \abs{\eta - \frac{2r}{s}}\leq\frac{1}{\abs{q}^2}=\frac{c_4(s)}{4s^2}.
    \end{equation*}
    On the other hand, \Cref{lem:theta,lem:Theta} yields
    \begin{equation*}
        \abs{\eta - \frac{2r}{s}}\geq\frac{c_4(s)}{\Delta s^2} 
    \end{equation*}
    where
    \begin{equation*}
        \Delta = 2+\sqrt{4-tm^{-2}}.
	\end{equation*}
	Combining the two inequalities, we obtain $\Delta\geq 4$, which is impossible since  $\Delta=2+\sqrt{4-tm^{-2}}<4$. This contradiction shows that \eqref{eq:contradiction_inequality} has no solutions.
\end{proof}

Now we will focus in building uncountably many complex numbers with real part $1/2$ satisfying \eqref{eq:1/2|q|^2}. These numbers will be obtained by taking limits of roots of Schmidt's $C$--minimal forms along the tree of solutions of Vulakh's equation. The following lemmas were proved by Schmidt in \cite[Section 5]{Sch}, imitating the analogous development done by Cassels in \cite[Chapter II]{Cassels}. For example, similar to the proof of \cite[Chapter II, Lemma 12]{Cassels}, we have the following \cite[Lemma 5.11]{Sch}.

\begin{lemma}\label{lem:Schmidt511}
    Let $f(x,y)=x^2+Bxy+Cy^2$, $B,C\in\R$. Then
    \begin{enumerate}
        \item If $f(l_2,\Lambda_2)\leq -1$ and $f(l_2-4\Lambda_2,\Lambda_2)\leq -1$, then $\Disc(f)\geq 16+4/\Lambda_2^2$.
        \item If $f(m_2,M_2)\leq -2$ and $f(m_2-4M_2,M_2)\leq -2$, then $\Disc(f)\geq 16+4/M_2^2$.
    \end{enumerate}
\end{lemma}

\begin{lemma}\label{lem:new_Schmidt511}
    Let $f(x,y)=x^2+Bxy+Cy^2$, $B,C\in\R$. Then
    \begin{enumerate}
        \item If $f(l_2,\Lambda_2)\leq -1$ and $f(m_2-4M_2,M_2)\leq -2$ with $\varepsilon=+1$, then $\Disc(f)\geq 16+4/\Lambda_2^2$.
        \item If $f(m_2,M_2)\leq -2$ and $f(l_2-4\Lambda_2,\Lambda_2)\leq -1$ with $\varepsilon=-1$, then $\Disc(f)\geq 16+4/M_2^2$.
    \end{enumerate}
\end{lemma}

\begin{proof}[Proof of \Cref{lem:new_Schmidt511}]
    We prove only the first statement, since the second is entirely analogous. Write
    \begin{align*}
        f(x,y)&=\left(x+\frac{1}{2}By\right)^2-\frac{1}{4}\Disc(f)y^2, \\
        F(x,y)&=\left(x+\frac{1}{2}\tilde{B} y\right)^2-\frac{1}{4}\Disc(F)y^2,
    \end{align*}
    where $F(x,y)=x^2+\tilde{B}xy+\tilde{C}y^2$ denotes either $F=g_{\Lambda_2}$ or $F=h_{M_2}$. We shall prove that $\Disc(f)-\Disc(g_{\Lambda_2})-8/\Lambda_2^2\geq 0$. Since
    \begin{equation*}
        \frac{f(x,y) - F(x,y)}{y^2} = \frac{\Disc(F)-\Disc(f)}{4} + \left(\frac{x}{y}+\frac{B}{2}\right)^2 - \left(\frac{x}{y}+\frac{\tilde{B}}{2}\right)^2,
    \end{equation*}
    it suffices to estimate the difference of squares.
    
    Suppose first that $F=g_{\Lambda_2}$. Taking $(x,y)=(l_2,\Lambda_2)$ and using \eqref{eq:special_values_of_g_and_h}, we have $f(x,y)\le F(x,y)-2$, and hence
    \begin{equation*}
        \Disc(f)-\Disc(g_{\Lambda_2})-\frac{8}{\Lambda_2^2}\geq 4\left(\frac{x}{y}+\frac{B}{2}\right)^2 - 4\left(\frac{x}{y}+\frac{\tilde{B}}{2}\right)^2.
    \end{equation*}
    where $\tilde B=4-2l_2/\Lambda_2$. Since $0\le l_2\le\Lambda_2$, we have $0\leq x/y\le1$. Therefore, if $B\ge\tilde B$, the right-hand side is non-negative, giving the desired conclusion.
    
    Next suppose that $F=h_{M_2}$. Taking $(x,y)=(m_2-4M_2,M_2)$, we obtain 
    \begin{equation*}
        \Disc(f)-\Disc(h_{M_2})-\frac{16}{M_2^2}\geq 4\left(\frac{x}{y}+\frac{B}{2}\right)^2 - 4\left(\frac{x}{y}+\frac{\tilde{B}}{2}\right)^2.
    \end{equation*}
    where $\tilde{B}=4-2m_2/M_2$. Since $0\leq m_2\leq M_2$, we get
    \begin{equation*}
        -\frac{x}{y} = \frac{4M_2-m_2}{M_2} > 3 > 2-\frac{m_2}{M_2} = \frac{\tilde{B}}{2}.
    \end{equation*}
    Therefore, if $B\leq\tilde{B}$, we would obtain that
    \begin{equation*}
        \Disc(f)\geq\Disc(h_{M_2})+\frac{16}{M_2^2}=16+\frac{8}{M_2^2}>16+\frac{4}{\Lambda_2^2},
    \end{equation*}
    as required. It remains to consider the case
    \begin{equation*}
        4-\frac{2m_2}{M_2}<B<4-\frac{2l_2}{\Lambda_2}.
    \end{equation*}
    However, this inequality implies
    \begin{equation*}
        0<\frac{m_2}{M_2}-\frac{l_2}{\Lambda_2}=\frac{\Lambda_2m_2-M_2l_2}{M_2\Lambda_2}=\frac{-\varepsilon M_1}{M_2\Lambda_2},
    \end{equation*}
    where the last equality follows from \eqref{eq:determinant_equation2}. Since $\varepsilon=1$, the right-hand side is negative, a contradiction.
\end{proof}

Let $f(x,y)$ be a binary quadratic form. We use Cassels notation, by writing
\begin{equation*}
    P(x,y): \text{ if } f(x,y)>0, \qquad N(x,y): \text{ if } f(x,y)<0.
\end{equation*}

\begin{lemma}
    Let $f(x,y)=x^2+Bxy+Cy^2$, $2\leq B\leq 4$, $0<\Disc(f)<16$. Suppose that 
    \begin{align*}
        \abs{f(x,y)}\geq 2 \text{ for all } (x,y)\in\Z_0^2, \\
        \abs{f(x,y)}\geq 1 \text{ for all } (x,y)\in\Z_1^2.
    \end{align*}
    Then $f$ is either a $g_\Lambda$ or an $h_M$.
\end{lemma}

The proof of this lemma actually yields a stronger statement. Combined with Schmidt's result from \cite{SchmidtCforms} that the forms $g_\Lambda$ and $h_M$ are pairwise distinct, it follows from the proof of \cite[Lemma 5.12]{Sch} that

\begin{lemma}\label{lem:N_inequalities}
    For $r=1,\dots,R$, let $(\Lambda_1^{(r)},\Lambda_2^{(r)};M_1^{(r)},M_2^{(r)})$ with associated quintuple $(\varepsilon^{(r)},l_1^{(r)},l_2^{(r)};m_1^{(r)},m_2^{(r)})$ be a downward path in the tree of \Cref{fig:Vulakh_tree}. Then for all $r=1,\dots,R-1$ we have
    \begin{equation}\label{eq:N_inequalities}
    \begin{aligned}
        N(l_1^{(r)},\Lambda_1^{(r)}) & &\text{and}\quad& N(m_1^{(r)}-4M_1^{(r)},M_1^{(r)}) &  \text{in case }\varepsilon^{(r)}=+1, \\
        N(l_1^{(r)}-4\Lambda_1^{(r)},\Lambda_1^{(r)}) & &\text{and}\quad& N(m_1^{(r)},M_1^{(r)}) &  \text{in case }\varepsilon^{(r)}=-1.
    \end{aligned}
    \end{equation}
\end{lemma}

Finally, to complete the proof of \Cref{thm:uncountably_many_for_c=2}, it suffices, by \Cref{thm:roots_inequality_theorem}, to pass to limits of roots of $C$--minimal forms along downward infinite paths, since the property \eqref{eq:1/2|q|^2} is closed. The argument is analogous to the proof of \cite[Chapter II, Lemma 14]{Cassels}.

\begin{theorem}
    Let $(\Lambda_1^{(r)},\Lambda_2^{(r)};M_1^{(r)},M_2^{(r)})$ with associated quintuple $(\varepsilon^{(r)},l_1^{(r)},l_2^{(r)};m_1^{(r)},m_2^{(r)})$, $r=1,2,\dots$ be a downward infinite path in the tree of \Cref{fig:Vulakh_tree}. Let $\eta^{(r)}>0$ (or $\eta^{(r)}<0$) be the positive (or negative) root of $g_{\Lambda_1^{(r)}}$. Then, the limit 
    \begin{equation}\label{eq:limit_eta}
        \eta = \lim_{r\to\infty}\eta^{(r)},
    \end{equation}
    exists and is distinct for each path. The same holds replacing $g_{\Lambda_1^{(r)}}$ by $h_{M_1^{(r)}}$.
\end{theorem}

\begin{proof}
    We first consider the family $h_{M_1^{(r)}}$. Since the roots of a quadratic polynomial depend continuously on its coefficients, it suffices to prove that the coefficients of
    \begin{equation*}
        h_{M_1^{(r)}}(x,y)=x^2+\left(4-\frac{2m_1^{(r)}}{M_1^{(r)}}\right)xy+\left(\frac{\mu_1^{(r)}-4m_1^{(r)}}{M_1^{(r)}}\right)y^2,
    \end{equation*}
    converge. Since neighboring solutions are related by \eqref{eq:neighbors_solutions}, it is enough to consider those indices $r$ for which $M_1^{(r+1)}=M_2^{(r)}$. By \eqref{eq:determinant_equation}
    \begin{equation*}
        \frac{m_1^{(r)}}{M_1^{(1)}}-\frac{m_2^{(r)}}{M_2^{(1)}}=\frac{2\varepsilon^{(r)}\Lambda_1^{(r)}}{M_1^{(r)}M_2^{(r)}}.
    \end{equation*}
    Using \eqref{eq:M_2>Lambda_1}, we obtain
    \begin{equation*}
        \abs{\frac{m_1^{(r)}}{M_1^{(1)}}-\frac{m_2^{(r)}}{M_2^{(1)}}}=\frac{2\Lambda_1^{(r)}}{M_1^{(r)}M_2^{(r)}}\leq \frac{2}{(M_1^{(r)})^2}.
    \end{equation*}
    If $M_1^{(r)}$ is eventually constant, then the forms $h_{M_1^{(r)}}$ eventually stabilize. Otherwise, the distinct values of $M_1^{(r)}$ form a strictly increasing sequence, and since $\sum_{n\ge1}n^{-2}$ converges, it follows that the coefficient $4-2m_1^{(r)}/M_1^{(r)}$ converges.
    
    Similarly, using \eqref{eq:lambda_i_mu_i} and $0\leq m_i\leq M_i$,
    \begin{align*}
        &\abs{\frac{\mu_1^{(r)}}{M_1^{(r)}}-\frac{\mu_2^{(r)}}{M_2^{(r)}}} \\
        &\leq\abs{\left(\frac{m_1^{(r)}}{M_1^{(r)}}\right)^2-\left(\frac{m_2^{(r)}}{M_2^{(r)}}\right)^2}+\abs{\frac{2}{(M_1^{(r)})^2}-\frac{2}{(M_2^{(r)})^2}}+\abs{\frac{4m_1^{(r)}}{M_1^{(r)}}-\frac{4m_2^{(r)}}{M_2^{(r)}}} \\
        &\leq 6\abs{\frac{m_1^{(r)}}{M_1^{(1)}}-\frac{m_2^{(r)}}{M_2^{(1)}}}+\frac{2}{(M_1^{(r)})^2}\leq\frac{14}{(M_1^{(r)})^2}.
    \end{align*}
    Hence the coefficient $(\mu_1^{(r)}-4m_1^{(r)})/M_1^{(r)}$ also converges.
    
    The proof for the family $g_{\Lambda_1^{(r)}}$ is completely analogous. Indeed,
    \begin{equation*}
        g_{\Lambda_1^{(r)}}(x,y)=x^2+\left(4-\frac{2l_1^{(r)}}{\Lambda_1^{(r)}}\right)xy+\left(\frac{\lambda_1^{(r)}-4l_1^{(r)}}{\Lambda_1^{(r)}}\right)y^2,
    \end{equation*}
    we have
    \begin{equation*}
        \abs{\frac{l_1^{(r)}}{\Lambda_1^{(1)}}-\frac{l_2^{(r)}}{\Lambda_2^{(1)}}}=\frac{2M_1^{(r)}}{\Lambda_1^{(r)}\Lambda_2^{(r)}}\leq \frac{2M_1^{(r)}}{\Lambda_1^{(r)}M_2^{(r)}}\leq\frac{2}{(\Lambda_1^{(r)})^2},
    \end{equation*}
    and the same argument shows that the coefficients of $g_{\Lambda_1^{(r)}}$ converge. Therefore, the limit \eqref{eq:limit_eta} exists for both roots of both families of quadratic forms.

    By \Cref{lem:N_inequalities}, the limit forms, which we denote by $G$ and $H$, satisfy the inequalities in \eqref{eq:N_inequalities}. If either of the sequences $\Lambda_1^{(r)}$ or $M_1^{(r)}$ is eventually constant, then the corresponding limit form is one of the forms $g_\Lambda$ or $h_M$, which are pairwise distinct by \cite{SchmidtCforms}. We may therefore assume that both sequences are unbounded. Since
    \begin{equation*}
        \Disc(g_{\Lambda_1^{(r)}})=16-\frac{4}{(\Lambda_1^{(r)})^2},\qquad \Disc(h_{M_1^{(r)}})=16-\frac{8}{(M_1^{(r)})^2},
    \end{equation*}
    it follows that $\Disc(G)=\Disc(H)=16$.

    Now let $(\tilde{\Lambda}_1^{(r)},\tilde{\Lambda}_2^{(r)};\tilde{M}_1^{(r)},\tilde{M}_2^{(r)})$ be a downward infinite path distinct from $(\Lambda_1^{(r)},\Lambda_2^{(r)};M_1^{(r)},M_2^{(r)})$. We shall prove that the corresponding limit forms $\tilde G$ and $\tilde H$ are different from $G$ and $H$. Let $R+1$ be the first index at which the two paths diverge, and suppose, by contradiction, that $G=\tilde G$ or $H=\tilde H$.

    By \eqref{eq:neighbors_solutions}, the three possible descendants of the common vertex 
    \begin{equation*}
        (\Lambda_1^{(R)},\Lambda_2^{(R)};M_1^{(R)},M_2^{(R)})=(\tilde{\Lambda}_1^{(R)},\tilde{\Lambda}_2^{(R)};\tilde{M}_1^{(R)},\tilde{M}_2^{(R)}),    
    \end{equation*} 
    are
    \begin{align*}
        (\Lambda_1^{(R+1)},-;M_2^{(R+1)},-) &\quad\text{ with quintuple }\quad (\varepsilon^{(R+1)};l_1^{(R+1)},-,m_2^{(R+1)},-), \\
        (\Lambda_2^{(R+1)},-;M_1^{(R+1)},-) &\quad\text{ with quintuple }\quad (\varepsilon^{(R+1)};l_2^{(R+1)},-,m_1^{(R+1)},-), \\
        (\Lambda_2^{(R+1)},-;M_2^{(R+1)},-) &\quad\text{ with quintuple }\quad (\varepsilon^{(R+1)};l_2^{(R+1)},-,m_2^{(R+1)},-).
    \end{align*}
    Since the two paths separate at this vertex, the form $G$ or $H$ must satisfy the inequalities in \eqref{eq:N_inequalities} for two distinct choices among the three descendants above. However, either of the first two choices is incompatible with the third, since \Cref{lem:Schmidt511} would imply that the corresponding form has discriminant strictly larger than $16$. Moreover, the first and second choices are incompatible by \Cref{lem:new_Schmidt511} by the same discriminant reason. This contradiction shows that the limit forms associated with distinct downward infinite paths are distinct.
\end{proof}

\section{Dynamical Lagrange spectra}\label{sec:dynamical_spectra}

The dynamical description of the classical Lagrange and Markov spectra made by Perron \cite{Perron} motivate the following definitions. Let $\varphi:X\to X$ be a continuous map defined on a metric space $X$ and $f:X\to \mathbb{R}$ be a continuous function. Given a point $x\in X$, we can define its Lagrange and Markov value as
\begin{eqnarray*}
   \ell_{\varphi,f}: X &\rightarrow& \mathbb{R} \\
   x &\mapsto& \ell_{\varphi,f}(x)=\limsup_{n\to \infty}f(\varphi^n(x)),
\end{eqnarray*}
\begin{eqnarray*}
   m_{\varphi,f}:X &\rightarrow& \mathbb{R} \\
   x &\mapsto& m_{\varphi,f}(x)=\sup_{n\to \infty}f(\varphi^n(x)).
\end{eqnarray*}
The Lagrange and Markov dynamical spectra associated with the data $(\varphi,f,X)$ are the sets
\begin{equation*}
    L_{\varphi,f,X}:=\{\ell_{\varphi,f}(x):x\in X\},
\end{equation*}
and
\begin{equation*}
    M_{\varphi,f,X}:=\{m_{\varphi,f}(x):x\in X\}.
\end{equation*}

\begin{example}
The classical Lagrange and Markov spectra can be realized as dynamical spectra where the dynamics is a smooth horseshoe over a surface. More precisely, define the natural extension of the Gauss map $T_1:(0,1]^2\to[0,1)^2$ by 
\begin{equation*}
    T_1(x,y)=\left(\frac{1}{x}-\left\lfloor \frac{1}{x}\right\rfloor, \frac{1}{\left\lfloor \frac{1}{x}\right\rfloor +y }\right),
\end{equation*}
and consider the height function $f(x,y)=\frac{1}{x}+y$. Nakada, Ito and Tanaka showed in \cite{NIT} that the map $T_1$ admits the smooth invariant density $d\mu=\frac{1}{(1+xy)^2}dxdy$. 

The dynamical Lagrange spectra with height $f$ associated to $T_1$ coincides with the real Lagrange spectrum up to a certain height. Indeed, if we denote the Gauss--Cantor set $C(N)=\{x=[0;a_1,a_2,\dots]:1\leq a_n\leq N\}$ for $N\in \mathbb{N}_{>0}$, then $\Lambda(N)=C(N)\times C(N)$ is a conservative horseshoe for $T_1$ and Moreira and Lima \cite{ML} proved that the dynamical Lagrange spectra of $T_1$ with respect to this horseshoe coincides with $L\cap (-\infty,\sqrt{N^2+4N}]$ for $N\geq 4$, where $L$ is the classical (real) Lagrange spectrum.
\end{example}

In general, we will consider the following setting (and we refer to Palis-Takens book \cite{PT} and Cerqueira-Matheus-Moreira paper \cite{CerqueiraMatheusGugu} for more details). Let $S$ be a surface and $\varphi:S\to S$ be a $C^2$ diffeomorphism possessing a horsehoe $\Lambda$ (i.e., a non-empty compact invariant hyperbolic set of saddle type which is transitive, locally maximal, and not reduced to a periodic orbit). Given $x\in\Lambda$, we denote by $e_x^u$ and $e_x^s$ unit vectors in the unstable and stable bundles of the hyperbolic decomposition $T_xS=E_x^u\oplus E_x^s$ of $\Lambda$, respectively. Let us fix a geometrical Markov partition $\{R_a\}_{a\in\mathcal{A}}$ with sufficiently small diameter consisting of rectangles $R_a\simeq I_a^s\times I_a^u$ delimited by compact pieces $I_a^s$, resp. $I_a^u$, of stable, resp. unstable, manifolds of certain points of $\Lambda$. We define the subset $\mathcal{T}\subset\mathcal{A}^2$ of admissible transitions as the subset of pairs $(a_0, a_1)\in\mathcal{A}^2$ such that $\varphi(R_{a_0})\cap R_{a_1}\neq\emptyset$. In this way, the dynamics of $\varphi$ on $\Lambda$ is topologically conjugated to a Markov shift $\Sigma_{\mathcal{T}}\subset\mathcal{A}^{\mathbb{Z}}$ of finite type associated to $\mathcal{T}$. 

\begin{figure}[htb!]
\includegraphics{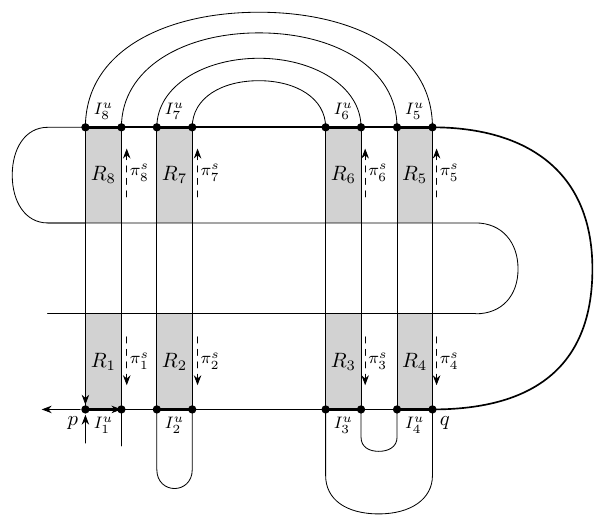}
\caption{Geometry of a horseshoe $\Lambda$. Notice that for the map $T_1$, the horizontal direction is expanding while the vertical is contracting.}
\end{figure}

The stable and unstable manifolds of $\Lambda$ can be extended to locally invariant $C^{1+\varepsilon}$-foliations in a neighborhood of $\Lambda$ for some $\varepsilon>0$. Therefore, we can use these foliations to define projections (also called holonomies) $\pi_a^u:R_a\to I_a^s\times\{i_a^u\}$ and $\pi_a^s: R_a\to \{i_a^s\}\times I_a^u$ of the rectangles into the connected components $I_a^s\times\{i_a^u\}$ and $\{i_a^s\}\times I_a^u$ of the stable and unstable boundaries of $R_a$ where $i_a^u\in\partial I_a^u$ and $i_a^s\in\partial I_a^s$ are fixed arbitrarily. Using these projections, we have the stable and unstable Cantor sets 
$$K^s=\bigcup\limits_{a\in\mathcal{A}}\pi_a^u(\Lambda\cap R_a) \quad \textrm{ and } \quad K^u=\bigcup\limits_{a\in\mathcal{A}}\pi_a^s(\Lambda\cap R_a)$$
associated to $\Lambda$. The stable and unstable Cantor sets $K^s$ and $K^u$ are $C^{1+\varepsilon}$-dynamically defined Cantor sets. 

\begin{definition}
    Let $\mathcal{A}$ be a finite alphabet, $\mathcal{B}\subset \mathcal{A}^2$, and $\Sigma$ the subshift of finite type of $\mathcal{A}^{\mathbb{Z}}$ with allowed transitions $\mathcal{B}$ which is topologically mixing, and such that every letter in $\mathcal{A}$ occurs in $\Sigma$. An expanding map of type $\Sigma$ is a map $\psi$ with the following properties:
	
	\begin{enumerate}
		\item[(1)] the domain of $\psi$ is a disjoint union $\sqcup_{\mathcal{B}}I(a,b)$, where, for each $(a,b)$, $I(a,b)$ is a compact subinterval of $I(a):=[0,1]\times \{a\}$;
		
		\item[(2)] for each $(a,b)\in \mathcal{B}$, the restriction of $\psi$ to $I(a,b)$ is a smooth diffeomorphism onto $I(b)$ satisfying $|D\psi(t)|>1$, for all $t$.
	\end{enumerate}
	The regular Cantor set associated to $\psi$ is defined as the maximal invariant set
	\begin{equation*}
	    K=\bigcap_{n\ge 0}\psi^{-n}\left(\sqcup_{\mathcal{B}}I(a,b)\right).
	\end{equation*}
\end{definition}

In a more familiar fashion, we can realize equivalently regular Cantor sets as hyperbolic sets on the real line (see \cite[Remark 2.3]{Phasetransition}). 

\begin{definition}
We say that a dynamically defined Cantor set $K$ is \emph{non-essentially affine} if there is no global conjugation $h\circ\psi\circ h^{-1}$ such that all branches $(h\circ\psi\circ h^{-1})|_{h(I_j)}$, $j=1,\dots,r$ are affine maps of the real line. 
\end{definition}

The importance of the previous definition is because of the following \cite[Theorem 1]{CartesianProduct}.

\begin{theorem}[Moreira's dimension formula]\label{thm:gugu_dimension_formula}
    Let $K,K^\prime\subset\R$ be dynamically defined Cantor sets such that $K$ is of class $C^2$ and non essentially affine. Then, given a $C^1$ map $f$ defined on a neighborhood of $K\times K^\prime$ to $\R$ such that, in some point of $K\times K^\prime$ its gradient is not parallel to any of the two coordinate axis. Then we have \[\dim_H(f(K\times K^\prime))=\min\{1,\dim_H(K)+\dim_H(K^\prime)\}.\]
\end{theorem}

Through this paper we will only be interested in Cantor sets defined by iterates of the Gauss map.

\begin{definition}\label{def:Gauss_Cantor} 
Given $B=\{\beta_1,\dots,\beta_\ell\}$, $\ell\geq 2$, a finite alphabet of finite words $\beta_j\in(\mathbb{N}_{>0})^{r_j}$, which is primitive (in the sense that $\beta_i$ does not begin by $\beta_j$ for all $i\neq j$) then the Gauss-Cantor set $K(B)\subseteq [0,1]$ associated with $B$ is defined as 
\begin{equation*}
    K(B):= \{[0;\gamma_1, \gamma_2, \dots] \ \mid\ \gamma_i\in B\}.
\end{equation*}
\end{definition}

\begin{proposition}{\cite[Proposition 1]{M1}}\label{prop:geometric_properties_prop1}
Gauss-Cantor sets are non-essentially affine.
\end{proposition}

Now we recall some general results on dynamical spectra associated with horseshoes on surfaces. The following theorems are not stated in their most general form, but rather in a simplified shape that is easier to use for applications.

Let $\mathcal{R}_{\varphi,\Lambda}$ denote the $C^r$--open and $C^r$--dense subset of $C^r(S,\R)$ defined by
\begin{equation*}\label{eq:R_varphi_Lambda}
    \mathcal{R}_{\varphi,\Lambda}
    =
    \{f\in C^r(S,\R): Df_x(e_x^u)\neq 0 \text{ and } Df_x(e_x^s)\neq 0,\ \forall x\in\Lambda\}.
\end{equation*}
We also let $\tilde{\mathcal{U}}$ denote the set of $\varphi\in\Diff^2_\omega(S)$, where $\omega$ is a smooth volume form, possessing a horseshoe $\Lambda$ with the following property: for every subhorseshoe $\tilde{\Lambda}\subset\Lambda$ and every function $g\in C^1(S,\R)$ satisfying $Dg_x(e_x^u)\neq0$ and $Dg_x(e_x^s)\neq0$ for some $x\in\tilde{\Lambda}$, one has
\[
\dim_H(g(\tilde{\Lambda}))=\min\{1,\dim_H(\tilde{\Lambda})\}.
\]

By \Cref{prop:geometric_properties_prop1}, the map $T_1$ belongs to $\tilde{\mathcal{U}}$. Indeed, every subhorseshoe of $\Lambda(N)$ decomposes, by the spectral theorem, into subhorseshoes that are products of stable and unstable Cantor sets. These Cantor sets are precisely Gauss--Cantor sets, which are non--essentially affine by \Cref{prop:geometric_properties_prop1}. Therefore, \Cref{thm:gugu_dimension_formula} yields the required dimension formula. For more explicit and more general sufficient conditions ensuring that a horseshoe belongs to $\tilde{\mathcal{U}}$, see \cite[Remark 2.18]{CerqueiraMatheusGugu} and \cite{CartesianProduct}.

It turns out that the fractal geometric properties of dynamical Markov and Lagrange spectra over a horseshoe $\Lambda$ are intimately related to the properties of the one parameter family of maximally invariant hyperbolic sets
\begin{equation*}
    \Lambda_t := \{x:\sup_{n\in\Z}f(\varphi^n(x))\leq t\},
\end{equation*}
which in turn depend on the stable and unstable sets
$$K^s_t=\bigcup\limits_{a\in\mathcal{A}}\pi_a^u(\Lambda_t\cap R_a) \quad \textrm{ and } \quad K^u_t=\bigcup\limits_{a\in\mathcal{A}}\pi_a^s(\Lambda_t\cap R_a).$$

The following proposition is fundamental for estimating the dimension of Markov and Lagrange spectra. It is analogous to \cite[Proposition 3]{LMV} which in turn is based on the proof of the main result of \cite{SergioGuguETDS}. 

\begin{proposition}\label{prop:dimension_of_subimages}
	Let $\varphi\in\tilde{\mathcal{U}}$. Then for every subhorseshoe $\tilde{\Lambda}\subset \Lambda$ and $f\in \mathcal{R}_{\varphi,\Lambda}$, we have $$\min\{1,\dim_H(\tilde{\Lambda})\}=\dim_H(\ell_{\varphi,f}(\widetilde{\Lambda}))=\dim_H(m_{\varphi,f}(\widetilde{\Lambda})).$$
\end{proposition}

For completeness, we will give the proof of \Cref{prop:dimension_of_subimages}. We need the following consequence of \cite[Proposition 1]{LMV}.

\begin{proposition}\label{prop:prop_1_gugu_villamil}
Let $\varphi\in\Diff_{\omega}^2(S)$ with a transitive horseshoe $\Lambda$. Fix $f\in \mathcal{R}_{\varphi,\Lambda}$ and $t\in \mathbb{R}$ such that $\dim_H(\Lambda_t)>0$. Then, for every $0<\eta<1$ there exists $\delta>0$ and a complete subshift 
$\Sigma(\mathcal{B}) \subset \Sigma \subset \mathcal{A}^{\mathbb{Z}}$, associated to a finite set $\mathcal{B}=\{\beta_1,\beta_2,\dots,\beta_m\}$ of finite words, such that
\begin{equation*}
    \Sigma(\mathcal{B}) \subset \Sigma_{t-\delta}
\end{equation*}
and 
\begin{equation*}
    \dim_H(\Lambda(\Sigma(\mathcal{B})))>(1-\eta)\dim_H(\Lambda_t)
\end{equation*}
where $\Lambda(\Sigma(\mathcal{B}))$ is the corresponding subhorseshoe given by the symbolic dynamic conjugation.
\end{proposition}
\begin{remark}
    In fact, \cite[Propositon 1]{LMV} is more general since it construct Bernoulli regular Cantor subsets of $K_u^t$ and $K_s^t$ respectively, which approximate their Hausdorff and box dimensions, without assuming the preservation of some area form.
\end{remark}

\begin{proof}[Proof of \Cref{prop:dimension_of_subimages}]

The inequalities \[\dim_H(\ell_{\varphi,f}(\tilde{\Lambda}))\leq \dim_H(m_{\varphi,f}(\tilde{\Lambda}))\leq \min\{1,\dim_H(\tilde{\Lambda})\}\] are immediate due to the fact that $L_{\varphi,f,\tilde{\Lambda}}\subset M_{\varphi,f,\tilde{\Lambda}}$.  

For the converse, given a subhorseshoe $\tilde{\Lambda}\subset \Lambda$, we use \Cref{prop:prop_1_gugu_villamil} for $\tilde{\Lambda}$ instead of $\Lambda$ and $t=\max_{x\in \tilde{\Lambda}}f(x)$, so we have that for every $\eta>0$, that there exists $\delta>0$ and a subhorseshoe $\tilde{\Lambda}^{\eta}\subset \tilde{\Lambda}$ such that $\dim_H(\tilde{\Lambda}^{\eta})>(1-\eta)\dim_H(\tilde{\Lambda})$ and $\max_{x\in \tilde{\Lambda}^{\eta}}f(x)\leq t-\delta$. By the compacity of $f^{-1}(t)\cap \tilde{\Lambda}$ and continuity of $Df$, we can choose $\delta$ in such a way that for $x\in \tilde{\Lambda}$ such that $f(x)>t-\delta$, $Df_x(e^{s,u}_x)\neq 0$. From now on we fix some $z\in \Lambda$ such that $f(z)=t$. Let $z=h(\cdots,a_{-n_0-1},a_{-n_0},\cdots,a_{-1};a_0,a_1,\cdots,a_{n_0},a_{n_0+1},\cdots)$ be the symbolic representation of $z$, where $n_0$ is chosen in a way such that for every point $x$ of the form $x=h(\cdots,\theta_{-n_0-2},\theta_{n_0-1},\alpha,\theta_{n_0+1},\theta_{n_0+2},\cdots), \alpha=(a_{-n_0-1},a_{-n_0},\cdots,a_{-1};a_0,a_1,\cdots,a_{n_0})$, we have that $f(x)>t-\frac{\delta}{2}$. Because of continuity of $Df$ and of the bundles $E^s,E^u$, choosing $\delta,$  we have that for every $x\in \tilde{\Lambda}$ with $f(x)>t-\delta$, $Df_x(e^{s,u}_x)\neq 0$. Considering now a periodic point $p\in \tilde{\Lambda}^{\eta}$ with symbolic representation of the form $p=h(\overline{d})=h(\cdots,d;d,d,\cdots$, for some $d=(d_0,\cdots,d_k)$, we have that, for some large enough $m_0$ and $m>m_0$, for every $q\in \tilde{\Lambda}^{\eta}, q=(\cdots,\theta_{-n},\theta_{-n+1},\cdots, \theta_{-1};\theta_0,\theta_1,\cdots,\theta_{n-1},\theta_n,\cdots)$, the Markov value of $A(q)=h(\cdots,\theta_{-2},\theta_{-1},d^m,\alpha,d^m,\theta_0,\theta_1,\cdots)$ will be realized in some position $j_0$ in the central block $(d^{m_0},\alpha,d^{m_0})$ and $f(\varphi^l(A(q)))<t-\delta/2$ in every position that does not correspond to this block. Notice that, because $\tilde{\Lambda}^{\eta}$ comes from a complete subshift, these bi-infinite sequences are allowed. Let $M$ be such that, if $|j|>M$, then the position $j$ is outside the block $(d^{m_0},\alpha,d^{m_0})$. Using the lemma for the function $F=\max_{|j|\leq M}f\circ\varphi^j$ and the space $\tilde{\Lambda}$, we get that there are words $\beta_1,\beta_2$ and some $|j_0|<M$ such that for any $y$ of the form $y=h(\cdots,\theta_{-2},\theta_{-1},\beta_1,d^{m},\alpha,d^m,\beta_2,\theta_0,\theta_1,\cdots)$,
    $$F(y)=f(\varphi^{j_0}(y))=f(h(\sigma^{j_0}
(\cdots,\theta_{-2},\theta_{-1},\beta_1,d^{m},\alpha,d^m,\beta_2,\theta_0,\theta_1,\cdots)))$$
for a bi-infinite sequence $(\theta_n)$ in the shift space of $\tilde{\Lambda}$ such that the transitions $\theta_{-1},\beta_1$ and $\beta_2,\theta_0$ are allowed. 

    Now, for a point $q\in \tilde{\Lambda}^{\eta}$, $q=h((\theta_n)_{n\in \mathbb{Z}})$, we define the diffeomorphism $\tilde{A}(q)=\varphi^{j_0}(h(\cdots,\theta_{-2},\theta_{-1},\beta_1,d^{m},\alpha,d^m,\beta_2,\theta_0,\theta_1,\cdots))$ defined on a neighborhood of $\tilde{\Lambda}^{\eta}$. This diffeo preserves the unstable and stable directions. We claim that $f(\tilde{A}(\tilde{\Lambda}^{\eta}))\subset L_{\varphi,f,\tilde{\Lambda}}$. Indeed, for any $(\theta_n)$ in the shift of $\tilde{\Lambda}$, if we define the blocks 
    $$\tau^{n}=(\theta_{-n},\theta_{-n+1},\cdots,\theta_{-1},\beta_1,d^m,\alpha,d^m,\beta_2,\theta_0,\theta_1,\cdots,\theta_{n}),$$
    and consider the point $q'=h(\cdots,\tau^{n},\cdots,\tau^1,\tau^0,\tau^1,\cdots,\tau^n,\cdots)$, by what we have just seen above, the maximum value of $f\circ \varphi^k(q')$ for the positions $k$ of some block $\tau^n$ will be attained in the position of the block corresponding to $j_0$. This will imply that $\ell_{f,\varphi}(q')=f(\tilde{A}(q))$. As $f(\tilde{A}(q))>t-\delta/2$, we have that $Df_{\tilde{A}(q)}(e^{s,u}_{\tilde{A}(q)})\neq 0$, so we have that the dimension formula $\dim_H(f(\tilde{A}(\tilde{\Lambda}^{\eta})))=\min\{1,\dim_H(\tilde{\Lambda}^{\eta})\}$ holds.

    This discussion implies that 
    \begin{align*}
        \dim_H(\ell_{\varphi,f}(\tilde{\Lambda}))\geq \dim_H(f(\tilde{A}(\tilde{\Lambda}^{\eta})))&=\min\{1,\dim_H(\tilde{\Lambda}^{\eta})\}\\
        &\geq \min\{1,(1-\eta)\dim_H(\tilde{\Lambda})\}.
    \end{align*}
    Making $\eta\to 0$, we have that $\dim_H(\ell_{\varphi,f}(\tilde{\Lambda}))\geq \min\{1,\dim_H(\tilde{\Lambda})\}$, which concludes the proof.
\end{proof}

Using the above tools, in \cite{SecondLagrangeSpectra} we proved the following theorem, generalizing the main result of \cite{LMV}.

\begin{theorem}\label{thm:continuity}
	Let $\varphi\in\tilde{\mathcal{U}}$ and and $f:\Lambda\to \mathbb{R}$ be a function of the form $f(x)=\max_{1\leq j\leq N}f_j(x)$ where each $f_j$ is a $C^r$ real valued function defined in a neighborhood of $\Lambda$ such that for every $x\in \Lambda$, it holds $Df_j(x)(e^{u}_x)\neq 0$ and $Df_j(x)(e^{s}_x)\neq 0$. Then the functions:
	$$t\mapsto d_u(t):= \dim_H(K^u_t) \ \mbox{and} \ t\mapsto d_s(t):=\dim_H(K^s_t)$$
	are continuous and in fact, they are equal. Moreover 
	\begin{equation*}
		\dim_H(K_t^u)=\overline{\dim_{B}}(K_t^u), \quad \dim_H(K_t^s)=\overline{\dim_{B}}(K_t^s),
	\end{equation*}
	and
	$$\dim_H(\Lambda_t)=d_u(t)+d_s(t)=2d_u(t)$$ and 
	$$\min\{1,\dim_H(\Lambda_t)\}=L(t).$$
\end{theorem}

Similarly, using the same construction of \cite{SecondLagrangeSpectra}, one can extend a result of S. Vieira \cite[Page 20]{SandoelThesis} for maximums of appropriate $C^1$ functions. Notice that the conservative hypothesis (i.e., the preservation of some area form $\omega$) is not necessary in the following result.

\begin{theorem}\label{thm:perfect_set}
	Let $\Lambda$ be a horseshoe of a $C^2$ diffeomorphism $\varphi$ in a surface $S$. Then the set 
    \begin{equation*}
        H_\Lambda = \{f\in C^1(S,\R):Df_z(e_z^u)\neq 0 \text{ or } Df_z(e_z^s)\neq 0, \forall z\in \Lambda\}
    \end{equation*}
    is an open and dense in $C^1(S,\R)$, such that for all $f=\max_{i\leq N}$, $f_i\in H_\Lambda$ we have $L_{\varphi,f,\Lambda}^{\prime\prime}=L_{\varphi,f,\Lambda}^{\prime}$ and
	\begin{equation*}
		\inf L_{\varphi,f,\Lambda}^{\prime} = \max\{t\in\R: \dim_H(L_{\varphi,f,\Lambda}\cap(-\infty,t))=0\}.
	\end{equation*}
\end{theorem}

\begin{proof}
    Define $\hat{S}=S\times \{1,\cdots,N\}$, $\hat{\Lambda}=\Lambda\times \{1,\cdots,N\}$ and $\hat{\varphi}:\hat{S}\to \hat{S}$ by \[\hat{\varphi}(x,j)=\begin{cases}
    (x,j+1), \text{ if} & 1\leq j\leq N-1\\
    (\varphi(x),1), \text{ if} & j = N
    \end{cases}\]

    This is a Rokhlin tower extension of $\Lambda$. Then $\hat{\Lambda}$ is also a horseshoe for $\hat{\varphi}$ on the surface $\hat{S}=S\times \{1,\cdots,N\}$ with the same hyperbolic splitting as $\Lambda$. We then define the function $\hat{f}:\hat{\Lambda}\to \mathbb{R}$ by $\hat{f}(x,j)=f_j(x)$. It is easy to check that for $(x,j)\in \hat{\Lambda}$, we have that 
    $$\limsup
    _{n\to +\infty}\hat{f}(\hat{\varphi}^n(x,j))=\limsup
    _{k\in \mathbb{Z}}\max_{1\leq l\leq N}f_l(\varphi^k(x))=\limsup
    _{n\to +\infty}f(\varphi^n(x))$$ and $$\sup
    _{n\in \mathbb{Z}}\hat{f}(\hat{\varphi}^n(x,j))=\sup
    _{k\in \mathbb{Z}}\max_{1\leq l\leq N}f_l(\varphi^k(x))=\sup
    _{n\in \mathbb{Z}}f(\varphi^n(x)).$$
    
    This implies that $L_{\varphi,f,\Lambda}=L_{\hat{\varphi},\hat{f},\hat{\Lambda}}$ and $M_{\varphi,f,\Lambda}=M_{\hat{\varphi},\hat{f},\hat{\Lambda}}$ and therefore the theorem is corollary of the results of \cite[Page 20]{SandoelThesis}.
    
\end{proof}

\section{Dynamical characterization for $L_{\frac{1}{2}+i\R}$}\label{sec:dynamical_characterization}

Now we are going to give a dynamical description to $L_{\frac{1}{2}+i\R}$ similar to what is done in the classical Lagrange spectrum. The main difference is that the base space is no longer a bi-infinite shift, but rather a skew-product over a bi-infinite shift.

Recall that A. Schmidt defined
\begin{equation*}
    k(\xi):=\underset{p,q\in\Z\left[i\right]}{\underset{|p|,|q|\rightarrow \infty}{\limsup}}(|q||q\xi-p|)^{-1}.
\end{equation*}
Since for numbers $\xi$ with $\Re(\xi)=1/2$ and $k(\xi)>2$ all good approximations also have real part $1/2$ by \Cref{lem:fundamental_lemma}, it motivates to define a more restrictive approximation constant
\begin{equation*}
    k_{\frac{1}{2}}(\xi):=\limsup_{\substack{p,q\in\Z\left[i\right] \\|p|,|q|\rightarrow \infty \\ \Re(p/q)=1/2}}(|q||q\xi-p|)^{-1}.
\end{equation*}

\begin{corollary}\label{cor:simple_cor}
    If $\alpha\in \R\setminus\Q$, is such that $k\left(\frac{1}{2}+i\alpha\right)>2$, then
    \begin{equation*}
        k\left(\frac{1}{2}+i\alpha\right)=k_{\frac{1}{2}}\left(\frac{1}{2}+i\alpha\right).
    \end{equation*}
\end{corollary}

For complex numbers $\frac{1}{2}+i\alpha$, $\alpha\in\R\setminus\Q$, it is possible to give a formula for $k_{\frac{1}{2}}(\frac{1}{2}+i\alpha)$ (see Equation (6.22) in \cite[Page 83]{Sch}) based on $C$--regular expansions of $\alpha$.

\begin{remark}\label{rem:optimal_inequality}
The constant of \Cref{lem:fundamental_lemma} and \Cref{cor:simple_cor}
can not be improved. For example, if we chose $\alpha=\frac{1}{2}(\sqrt{2}-1)$, which is half the positive root of the quadratic form $h_1$ defined on \eqref{eq:forms}, then for any $\varepsilon>0$ there are infinitely many $p,q\in\Z[i]$ with $\Re(p/q)\neq 1/2$ but such that
\begin{equation}\label{eq:counterexample}
    \abs{\frac{1}{2}+i\alpha-\frac{p}{q}}<\frac{1}{(2-\varepsilon)|q|^2}.
\end{equation}
In fact, we can prove that for any $p,q\in\Z[i]$ it holds that
\begin{equation}\label{eq:lower_inequality}
    \frac{1}{2|q|^2}<\abs{\frac{1}{2}+i\alpha-\frac{p}{q}}.
\end{equation}
As consequence $k\left(\frac{1}{2}+i\alpha\right)=2$. It is also possible to show that 
\begin{equation*}
    k_{\frac{1}{2}}\left(\frac{1}{2}+i\alpha\right)=\sqrt{2},
\end{equation*}
based on the formula of Schmidt (see \cite[Example 4.2]{Sch} and \cite[Equation (6.22)]{Sch}).

Let us give the proof of \eqref{eq:lower_inequality} first. Notice that we can clearly assume that $\gcd(p,q)=1$. Because of \Cref{lem:fundamental_lemma}, we have that if $\abs{\frac{1}{2}+i\alpha-\frac{p}{q}}\leq\frac{1}{2|q|^2}$, then $\Re(p/q)=1/2$ and there are some constraints for $\Im(p/q)$. Let us denote by $(p_m/q_m)_{m\geq 0}$ the convergents of the continued fraction of $\alpha=\frac{1}{2}(\sqrt{2}-1)=[0;\overline{4,1}]=[0;a_1,a_2,\dots]$. Since $q_m\equiv 0\pmod{4}$ for $m$ odd and $q_m\equiv 1\pmod{4}$ for $m$ even, we have that the case $q_{m+1}\pm q_m \equiv 2\pmod{4}$ of \Cref{lem:fundamental_lemma} never occurs. Therefore, we can conclude that there is an $n\geq 0$ such that $\Im(p/q)=p_n/q_n$. Since for $n=0$ we have $p/q=0$ and $q=1$, we can assume that $n\geq 1$. Depending on the parity of $n$, the expression for $q$ will change according to \eqref{eq:denominator_cases}. If $n$ is odd, then $q=q_n$ so
\begin{align*}
    \abs{q}^{-2}\abs{\frac{1}{2}+i\alpha-\frac{p}{q}}^{-1}&=q_n^{-2}\abs{\alpha-\frac{p_n}{q_n}}^{-1} \\
    &=\alpha_{n+1}+\beta_{n+1} \\
    &=[1;\overline{4,1}]+[0;(4,1)^{(n-1)/2},4] \\
    &<[1;\overline{4,1}]+[0;4]=\frac{1}{2}(1+\sqrt{2})+\frac{1}{4}<2.
\end{align*}
If $n$ is even, then $q=2q_n$ so
\begin{align*}
    \abs{q}^{-2}\abs{\frac{1}{2}+i\alpha-\frac{p}{q}}^{-1}&=4q_n^{-2}\abs{\alpha-\frac{p_n}{q_n}}^{-1} \\
    &=\frac{\alpha_{n+1}+\beta_{n+1}}{4} \\
    &=\frac{[4;\overline{1,4}]+[0;(1,4)^{n/2}]}{4}<\frac{6}{4}=\frac{3}{2}<2.
\end{align*}

A sequence of Gaussian integers that are very good approximations for $\frac{1}{2}+i\alpha$, are the sequences  $p^{(n)},q^{(n)}\in\Z[i]$ defined for all $n\geq 0$ by
\begin{equation*}
    \begin{pmatrix}
        p^{(n)} \\
        q^{(n)}
    \end{pmatrix}=
    \begin{pmatrix}
        1 & -1+i \\
        2 & -1+2i
    \end{pmatrix}^n
    \begin{pmatrix}
        0 \\
        i
    \end{pmatrix}.
\end{equation*}
The motivation to define these approximations is because they are \emph{complex convergents} \cite[Page 15]{Sch} defined by Schmidt. Notice $\frac{1}{2}+i\alpha$ is a fixed point of the M\"{o}bius map
\begin{equation*}
    z\mapsto \frac{z+(-1+i)}{2z+(-1+2i)}.
\end{equation*} 
First, let us give a proof of \eqref{eq:counterexample} based on Schmidt's results. According to \cite[Theorem 2.7]{Sch}, if we denote
\begin{equation*}
    \begin{pmatrix}
        a & b \\
        c & d
    \end{pmatrix} =
    \begin{pmatrix}
        1 & -1+i \\
        2 & -1+2i
    \end{pmatrix}
\end{equation*}
we have that $k\left(\frac{1}{2}+i\alpha\right)=\frac{\sqrt{\abs{D}}}{\mu}$ where $D=(d-a)^2+4bc=-8$ is the discriminant of the quadratic form $F(X,Y)=cX^2+(d-a)XY-bY^2=2X^2+(-2+2i)XY+(1-i)Y^2$ and 
\begin{multline*}
    \mu=\min\{\abs{F(1,1-i)},\abs{F(0,i)},\abs{F(1,1)},\abs{F(1+2i,2)}, \\
    \abs{F(-1-i,-1)},\abs{F(i,i)}\}=\sqrt{2}.
\end{multline*}
In particular, we have $k\left(\frac{1}{2}+i\alpha\right)=\frac{\sqrt{8}}{\sqrt{2}}=2$. On the other hand, since $k_{\frac{1}{2}}\left(\frac{1}{2}+i\alpha\right)=\sqrt{2}<2$, the limsup is equal to $k\left(\frac{1}{2}+i\alpha\right)=\limsup_{|p|,|q|\to\infty}(|q||q\alpha-p|)^{-1}$ and will be attained at fractions $p/q$ with $\Re(p/q)\neq1/2$. This proves \eqref{eq:counterexample}.

Now, we will give an elementary proof of \eqref{eq:counterexample}. We claim that the following relations hold for all $n\geq0$,
\begin{equation}\label{eq:auxiliary_identities}
    4p_{2n}=q_{2n-1}, \quad p_{2n+1}=q_{2n}, \quad p_{2n}+q_{2n}=p_{2n+2}, \quad 4p_{2n+2}^2+1=q_{2n+2}q_{2n},
\end{equation}
\begin{equation}\label{eq:q_2n}
    q_{2n}=\begin{cases}
        q_n^2/4+q_{n-1}^2, & \text{if $n$ is odd}, \\
        q_{n-1}^2/4+q_{n}^2, & \text{if $n$ is even},
    \end{cases}
\end{equation}
\begin{equation}\label{eq:q_2n+1}
    q_{2n+1}=\begin{cases}
        q_n^2+2q_nq_{n-1}, & \text{if $n$ is odd}, \\
        4q_{n}^2+2q_{n}q_{n-1}, & \text{if $n$ is even},
    \end{cases}
\end{equation}

Recall that by definition $q_{0}=1, q_{-1}=0, q_{-2}=1$. For $n=0$, all the identities are just trivial computation. For the rest, we use induction: 
\begin{align*}
    4p_{2n+2}&=4(q_{2n}+p_{2n})=4q_{2n}+q_{2n-1}=q_{2n+1},  \\
    p_{2n+3}&=4p_{2n+2}+p_{2n+1}=q_{2n+1}+q_{2n}=q_{2n+2}, \\
    p_{2n+4}&=p_{2n+3}+p_{2n+2}=q_{2n+2}+p_{2n+2},
\end{align*}
and
\begin{align*}
    4p_{2n+4}^2+1&=4(p_{2n+2}+q_{2n+2})^2+1 \\
    &=(4p_{2n+2}^2+1)+2q_{2n+2}(4p_{2n+2}+2q_{2n+2}) \\
    &=q_{2n+2}q_{2n}+2q_{2n+2}(q_{2n+1}+2q_{2n+2}) \\
    &= q_{2n+2}(4q_{2n+2}+2q_{2n+1}+q_{2n}) \\
    &=q_{2n+2}(q_{2n+3}+q_{2n+2})=q_{2n+2}q_{2n+4}.
\end{align*}
Now we will prove \eqref{eq:q_2n} and \eqref{eq:q_2n+1} simultaneously. If $n$ is odd, then
\begin{align*}
    q_{2n+2}=q_{2n+1}+q_{2n}&=q_n^2+2q_nq_{n-1}+q_n^2/4+q_{n-1}^2 \\
    &=(q_{n}+q_{n-1})^2+q_n^2/4=q_n^2/4+q_{n+1}^2,
\end{align*}
while if $n$ is even
\begin{align*}
    q_{2n+2}=q_{2n+1}+q_{2n}&=4q_n^2+2q_nq_{n-1}+q_{n-1}^2/4+q_n^2 \\
    &=(4q_n+q_{n-1})^2/4+q_n^2=q_{n+1}^2/4+q_n^2.
\end{align*}
If $n$ is odd
\begin{align*}
    4q_{n+1}^2+2q_{n+1}q_{n}&=4(q_n+q_{n-1})^2+2(q_n+q_{n-1})q_{n}  \\
    &=5(q_{n}^2+2q_nq_{n-1})+4q_{n-1}^2+q_n^2 \\
    &=5q_{2n+1}+4q_{2n} = 4q_{2n+2}+q_{2n+1}=q_{2n+3},
\end{align*}
while if $n$ is even
\begin{align*}
    q_{n+1}^2+2q_{n+1}q_{n}&=(4q_n+q_{n-1})^2+2(4q_n+q_{n-1})q_{n} \\
    &=5(4q_{n}^2+2q_nq_{n-1})+q_{n-1}^2+4q_n^2 \\
    &=5q_{2n+1}+4q_{2n} = 4q_{2n+2}+q_{2n+1}=q_{2n+3}.
\end{align*}

We claim that for all $n\geq 0$
\begin{equation}\label{eq:p(n)}
    p^{(n)}=\begin{cases}
        i^{n+1}q_{n-1}(1+i), & \text{if $n$ odd}, \\
        i^{n+1}\frac{q_{n-1}}{2}(1+i), & \text{if $n$ even.}
    \end{cases}
\end{equation}
\begin{equation}\label{eq:q(n)}
    q^{(n)}=\begin{cases}
        (-1)^{(n+1)/2}(q_{n}/2+q_{n-1}i), & \text{if $n$ odd}, \\
        (-1)^{n/2}(-q_{n-1}/2+q_{n}i), & \text{if $n$ even.} 
    \end{cases}
\end{equation}

For $n=0$ is direct computation. Notice that we have the recurrence relations $p^{(n+1)}=p^{(n)}+(i-1)q^{(n)}$ and $q^{(n+1)}=2p^{(n)}+(2i-1)q^{(n)}$ for all $n\geq 0$. In particular if both \eqref{eq:p(n)} and \eqref{eq:q(n)} hold up to $n$, then if $n=2k+1$ is odd
\begin{align*}
    p^{(n+1)}&=p^{(2k+1)}+(i-1)q^{(2k+1)} \\
    &= (-1)^{k+1}q_{2k}(1+i)+(-1)^{k+1}(i-1)(q_{2k+1}/2+q_{2k}i) \\
    &=(-1)^{k+1}(q_{2k+1}/2)(-1+i)=i^{n+2}(1+i)(q_{n}/2),
\end{align*}
while if $n=2k$ is even
\begin{align*}
    p^{(n+1)}&=p^{(2k)}+(i-1)q^{(2k)} \\
    &= (-1)^{k}(q_{2k-1}/2)(-1+i)+(-1)^{k}(i-1)(-q_{2k-1}/2+q_{2k}i) \\
    &=(-1)^{k}q_{2k}(-1-i)=i^{n+2}(1+i)q_{n}.
\end{align*}
Similarly, if $n=2k+1$ is odd
\begin{align*}
    q^{(n+1)}&=2p^{(2k+1)}+(2i-1)q^{(2k+1)} \\
    &=(-1)^{k+1}(2q_{2k})(1+i)+(-1)^{k+1}(2i-1)(q_{2k+1}/2+q_{2k}i) \\
    &=(-1)^{k}(q_{2k+1}/2)+(-1)^{k+1}(q_{2k+1}+q_{2k})i \\
    &=(-1)^{k}(q_{2k+1}/2)+(-1)^{k+1}q_{2k+2}i,
\end{align*}
while if $n=2k$ is even
\begin{align*}
    q^{(n+1)}&=2p^{(2k)}+(2i-1)q^{(2k)} \\
    &=(-1)^{k}q_{2k-1}(-1+i)+(-1)^{k}(2i-1)(-q_{2k-1}/2+q_{2k}i) \\
    &=(-1)^{k}(-2q_{2k}-q_{2k-1}/2)+(-1)^{k+1}q_{2k}i \\
    &=(-1)^{k+1}(q_{2k+1}/2)+(-1)^{k+1}q_{2k}i.
\end{align*}

Now we claim that for all $n\geq 0$
\begin{equation}\label{eq:recurrent_convergents}
    \frac{p^{(n)}}{q^{(n)}}=\frac{1}{2}+\frac{(-1)^{n+1}}{2q_{2n}}+\frac{p_{2n}}{q_{2n}}i.
\end{equation}
For $n=0$ is just a trivial computation. Now we will do induction again and \eqref{eq:auxiliary_identities}:
\begin{align*}
    \frac{p^{(n+1)}}{q^{(n+1)}}&=\frac{p^{(n)}+(i-1)q^{(n)}}{2p^{(n)}+(2i-1)q^{(n)}}=\frac{1}{2}-\frac{1}{2(2p^{(n)}/q^{(n)}+(2i-1))} \\
    &=\frac{1}{2}-\frac{q_{2n}}{2((-1)^{n+1}+2(p_{2n}+q_{2n})i)}=\frac{1}{2}-\frac{q_{2n}}{2((-1)^{n+1}+2p_{2n+2}i)} \\
    &=\frac{1}{2}+\frac{(-1)^{n}q_{2n}+2p_{2n+2}q_{2n}}{2(4p_{2n+2}^2+1)} =\frac{1}{2}+\frac{(-1)^{n}q_{2n}+2p_{2n+2}q_{2n}}{2q_{2n+2}q_{2n}} \\
    &=\frac{1}{2}+\frac{(-1)^{n}}{2q_{2n+2}}+\frac{p_{2n+2}}{q_{2n+2}}.
\end{align*}

Finally, from \eqref{eq:q(n)} and \eqref{eq:q_2n} it follows that $\abs{q^{(n)}}=\sqrt{q_{2n}}$. Therefore
\begin{align*}
    \abs{q^{(n)}}^2\abs{\frac{1}{2}+i\alpha-\frac{p^{(n)}}{q^{(n)}}}&=q_{2n}\abs{\frac{(-1)^{n}}{2q_{2n}}+i\left(\alpha-\frac{p_{2n}}{q_{2n}}\right)} \\
    &=\sqrt{\frac{1}{4}+\frac{1}{(\alpha_{2n+1}+\beta_{2n+1})^2q_{2n}^2}}
\end{align*}
where $\alpha_{2n+1}=[4;\overline{1,4}]$ and $\beta_{2n+1}=[0;(1,4)^{n}]$. Since $q_{2n}\to\infty$ as $n\to\infty$, given $\varepsilon>0$, we have that for all sufficiently large $n$ that
\begin{equation*}
    \frac{1}{2}<\abs{q^{(n)}}^2\abs{\frac{1}{2}+i\alpha-\frac{p^{(n)}}{q^{(n)}}}<\frac{1}{2-\varepsilon}.
\end{equation*}

This finishes the proof of \eqref{eq:counterexample}. 

\end{remark}

\begin{lemma}\label{lem:formula_for_k}
Let $\alpha\in\R\setminus\Q$. Then if $k(\frac{1}{2}+i\alpha)>2$ we have
\begin{equation*}
    k\left(\frac{1}{2}+i\alpha\right)=\lim_{|p|,|q|\to\infty}\abs{q}^{-2}\abs{\frac{1}{2}+i\alpha-\frac{p}{q}}^{-1}=\limsup_{n\to\infty}\eta(n)
\end{equation*}
where
\begin{equation*}
    \eta(n):=\begin{cases}
        \alpha_{n+1}+\beta_{n+1}, & \text{if $q_n\equiv 0\pmod{4}$}, \\
        2(\alpha_{n+1}+\beta_{n+1}), & \text{if $q_n\equiv 2\pmod{4}$}, \\
        \frac{\alpha_{n+1}+\beta_{n+1}}{4}, & \text{if $q_n$ odd and $q_{n-1}$ even}, \\
        \max\left\{\frac{\alpha_{n+1}+\beta_{n+1}}{4},2\left(\frac{1}{1+\beta_{n+1}}+\frac{1}{\alpha_{n+1}-1}\right)\right\}, &\shortstack[l]{
        if $q_n,q_{n-1}$ are odd \\
        and $q_n\equiv q_{n-1}\pmod{4}$},  \\
        \max\left\{\frac{\alpha_{n+1}+\beta_{n+1}}{4},2\left(\frac{1}{1-\beta_{n+1}}-\frac{1}{\alpha_{n+1}+1}\right)\right\}, &\shortstack[l]{
        if $q_n,q_{n-1}$ are odd \\
        and $q_n\not\equiv q_{n-1}\pmod{4}$}, 
    \end{cases}
\end{equation*}
and $\alpha_{n+1}=[a_{n+1};a_{n+2},a_{n+3}]$ and $\beta_{n+1}=[0;a_{n},a_{n-1},\dots,a_1]$. Moreover for any $\alpha\in\R\setminus\Q$ we have the inequality
\begin{equation}\label{eq:inequality_k_less_than_2}
    k\left(\frac{1}{2}+i\alpha\right)=\lim_{|p|,|q|\to\infty}\abs{q}^{-2}\abs{\frac{1}{2}+i\alpha-\frac{p}{q}}^{-1}\geq\limsup_{n\to\infty}\eta(n)
\end{equation}
\end{lemma}

\begin{proof}
First, notice that the limsup will be attained at pairs $\gcd(p,q)=1$. Since by hypothesis $k(\frac{1}{2}+i\alpha)>2$, we can assume further that $\abs{\frac{1}{2}+i\alpha-\frac{p}{q}}<\frac{1}{2\abs{q}^2}$. By \Cref{lem:fundamental_lemma} we have that for those fractions $\frac{p}{q}=\frac{1}{2}+i\frac{r}{s}$ are of the form \eqref{eq:denominator_cases} for some $q_n\leq s\leq q_{n+1}$. Conversely, for each positive integer $n$, we can define $\frac{p}{q}=\frac{1}{2}+i\frac{r}{s}$ with $\gcd(r,s)=1$ in either of the following forms: $\frac{r}{s}=\frac{p_n}{q_n}$ (and no restrictions) or $\frac{r}{s}=\frac{p_{n+1}-p_n}{q_{n+1}-q_n}$ with $q_{n+1}-q_n\equiv 2\pmod{4}$ or $\frac{r}{s}=\frac{p_{n}+p_{n-1}}{q_{n}+q_{n-1}}$ with $q_n+q_{n-1}\equiv 2\pmod{4}$. In particular, the value of $k\left(\frac{1}{2}+i\alpha\right)$ will be attained precisely at the $p/q$ of this three families. Now we will consider each one of these cases. 

Recall that $\gcd(q_n,q_{n-1})=1$. If $s\equiv 0\pmod{4}$, then $\frac{p}{q}=\frac{p_n}{q_n}$ and because of \eqref{eq:denominator_cases}, $q=q_n$, so $q_n=s$ is even and $q_{n-1}$ is odd. In particular 
\begin{equation*}
    \abs{q}^{-2}\abs{\frac{1}{2}+i\alpha-\frac{p}{q}}^{-1}=\alpha_{n+1}+\beta_{n+1}.
\end{equation*}
If $s\equiv 1\pmod{2}$ then $\frac{r}{s}=\frac{p_n}{q_n}$ and because of \eqref{eq:denominator_cases}, $q=2q_n$, so $q_n=s$ is odd. In this case we will have
\begin{equation*}
    \abs{q}^{-2}\abs{\frac{1}{2}+i\alpha-\frac{p}{q}}^{-1}=\frac{\alpha_{n+1}+\beta_{n+1}}{4}.
\end{equation*}

If $s\equiv 2\pmod{4}$, then we have that either 
\begin{itemize}
    \item $\frac{r}{s}=\frac{p_n}{q_n}$ and so $q=q_n/(1+i)$,
    \item $\frac{r}{s}=\frac{p_{n+1}-p_n}{q_{n+1}-q_n}$ and so $q=(q_{n+1}-q_n)/(1+i)$,
    \item $\frac{r}{s}=\frac{p_{n}+p_{n-1}}{q_{n}+q_{n-1}}$ and so $q=(q_{n}+q_{n-1})/(1+i)$.
\end{itemize}

In the first case $q_n=s\equiv2\pmod{4}$ and 
\begin{equation*}
    \abs{q}^{-2}\abs{\frac{1}{2}+i\alpha-\frac{p}{q}}^{-1}=2(\alpha_{n+1}+\beta_{n+1}).
\end{equation*}

Since $q_n$ and $q_{n-1}$ can not be even simultaneously, we have that $q_{n+1}-q_n=s\equiv 2\pmod{4}$ and $q_n+q_{n-1}=s\equiv 2\pmod{4}$ forces both denominators to be odd. In the second case we must have $q_{n+1}\not\equiv q_n\pmod{4}$ while in the third $q_{n}\equiv q_{n-1}\pmod{4}$. 

Writing $\alpha=\frac{\alpha_{n+2}p_{n+1}+p_n}{\alpha_{n+2}q_{n+1}+q_n}$, we have in the second case,
\begin{align*}
    \abs{\alpha-\frac{r}{s}}&=\abs{\frac{\alpha_{n+2}p_{n+1}+p_n}{\alpha_{n+2}q_{n+1}+q_n}-\frac{p_{n+1}-p_n}{q_{n+1}-q_n}} \\
    &=\frac{\alpha_{n+2}+1}{(\alpha_{n+2}q_{n+1}+q_n)(q_{n+1}-q_n)}=\frac{(\alpha_{n+2}+1)(1-\beta_{n+2})}{(\alpha_{n+2}+\beta_{n+2})(q_{n+1}-q_n)^2}
\end{align*}
hence
\begin{equation*}
    \abs{q}^{-2}\abs{\frac{1}{2}+i\alpha-\frac{p}{q}}^{-1}=\frac{2(\alpha_{n+2}+\beta_{n+2})}{(\alpha_{n+2}+1)(1-\beta_{n+2})}=2\left(\frac{1}{1-\beta_{n+2}}-\frac{1}{\alpha_{n+2}+1}\right).
\end{equation*}
In the third case,
\begin{align*}
    \abs{\alpha-\frac{r}{s}}&=\abs{\frac{\alpha_{n+1}p_{n}+p_{n-1}}{\alpha_{n+1}q_{n}+q_{n-1}}-\frac{p_{n}+p_{n-1}}{q_{n}+q_{n-1}}} \\
    &=\frac{\alpha_{n+1}-1}{(\alpha_{n+1}q_{n}+q_{n-1})(q_{n}+q_{n-1})}=\frac{(\alpha_{n+1}-1)(1+\beta_{n+1})}{(\alpha_{n+1}+\beta_{n+1})(q_{n}+q_{n-1})^2}
\end{align*}
hence
\begin{equation*}
    \abs{q}^{-2}\abs{\frac{1}{2}+i\alpha-\frac{p}{q}}^{-1}=\frac{2(\alpha_{n+1}+\beta_{n+1})}{(\alpha_{n+1}-1)(1+\beta_{n+1})}=2\left(\frac{1}{1+\beta_{n+1}}+\frac{1}{\alpha_{n+1}-1}\right).
\end{equation*}

Conversely, the formula \eqref{eq:denominator_cases} shows that indeed each of the possibilities in the formula for $\eta(k)$ can appear as errors of approximations with numbers $p/q$ whose real part is $\frac{1}{2}$. In particular, even when $k\left(\frac{1}{2}+i\alpha\right)\leq 2$, the limit $\limsup_{k\to\infty}\eta(k)$ is still a lower bound for the best approximation constant $k\left(\frac{1}{2}+i\alpha\right)$, which proves \eqref{eq:inequality_k_less_than_2}.
\end{proof}

The previous lemma motivates a definition that captures both information 
\begin{multline*}
    (q_{n-1}\pmod{4}, q_n\pmod{4}) \\
    \in\{(i,j):0\leq i,j<4 \text{ and }i, j\text{ not both even}\}:=\mathcal{I}.
\end{multline*}
by using $q_n=a_nq_{n-1}+q_{n-2}$ recursively.

Let $\tilde{\Sigma}:=(\N^*)^{\Z}\times\mathcal{I}$, consider the homeomorphism $\varphi:\tilde{\Sigma}\to\tilde{\Sigma}$ given by 
\begin{equation*}
    \varphi((a_n)_{n\in\Z},(\varepsilon_0,\varepsilon_1)):=((a_{n+1})_{n\in\Z},(\varepsilon_1,a_2\varepsilon_1+\varepsilon_0\pmod{4})),
\end{equation*}

and height function $f:\tilde{\Sigma}\to\R$
\begin{equation*}
    f((a_n)_{n\in\Z},(\varepsilon_0,\varepsilon_1)):=\begin{cases}
        \alpha_{2}+\beta_{2}, & \text{if $\varepsilon_1=0$}, \\
        2(\alpha_{2}+\beta_{2}), & \text{if $\varepsilon_1=2$}, \\
        \frac{\alpha_{2}+\beta_{2}}{4}, & \shortstack[l]{if $\varepsilon_1$ odd \\ and $\varepsilon_0$ even}, \\
        \max\left\{\frac{\alpha_{2}+\beta_{2}}{4},2\left(\frac{1}{1+\beta_{2}}+\frac{1}{\alpha_{2}-1}\right)\right\}, &\shortstack[l]{if $\varepsilon_1,\varepsilon_0$ odd \\ and $\varepsilon_0=\varepsilon_1$},  \\
        \max\left\{\frac{\alpha_{2}+\beta_{2}}{4},2\left(\frac{1}{1-\beta_{2}}-\frac{1}{\alpha_{2}+1}\right)\right\}, &\shortstack[l]{if $\varepsilon_1,\varepsilon_0$ odd \\ and $\varepsilon_0\neq\varepsilon_1$},  \\
    \end{cases}
\end{equation*}
where $\alpha_{2}=[a_{2};a_{3},\dots]$ and $\beta_{2}=[0;a_1,a_{0},\dots]$. Using this data, we can define the dynamical Lagrange spectrum
\begin{equation*}
    L_{\varphi,f,\tilde{\Sigma}}:=\{\ell_{\varphi,f}(\tilde{x})=\limsup_{n\to\infty}f(\varphi^n(\tilde{x}))<\infty:\tilde{x}\in\tilde{\Sigma}\}.
\end{equation*}

\begin{lemma}\label{lem:model_lemma}

For any $t\geq 2$ we have 
\begin{equation}\label{eq:equality_after_t_2}
    L_{\frac{1}{2}+i\R}\cap[2,t]=L_{\varphi,f,\tilde{\Sigma}}\cap[2,t].
\end{equation}

Moreover, the set $L_{\varphi,f,\tilde{\Sigma}}\cap(-\infty,2)$ is countable. In particular for any $t\in\R$
\begin{equation}\label{eq:equal_dimensions}
    \dim_H\left(L_{\frac{1}{2}+i\R}\cap (-\infty,t)\right)=\dim_H\left(L_{\varphi,f,\tilde{\Sigma}}\cap (-\infty,t)\right).
\end{equation}
\end{lemma}

\begin{proof}
Given $\gamma=[0;b_1,b_2,\dots]\in(0,1)\cap(\R\setminus\Q)$ with $k\left(\frac{1}{2}+i\gamma\right)>2$, defining the point 
\[\tilde{y}=((\dots,1,1,1,b_1,b_2,\dots),(1,b_1\pmod{4}))\in\tilde{\Sigma},\]
we will have by \Cref{lem:formula_for_k} that
\begin{equation*}
    k\left(\frac{1}{2}+i\gamma\right)=\limsup_{n\to\infty}\eta(n)=\ell(\tilde{y})\in L_{\varphi,f,\tilde{\Sigma}}.
\end{equation*}
Conversely, given $\tilde{x}=((a_n)_{n\in\Z},(\varepsilon_0,\varepsilon_1))\in\tilde{\Sigma}$, define inductively $\varepsilon_{n+1}\equiv a_{n+1}\varepsilon_n+\varepsilon_{n-1} \pmod{4}$ with $0\leq\varepsilon_n<4$ for $n\geq 1$. Let $N\geq 0$ be minimum such that $\varepsilon_N=1$ if such an $N$ exists, otherwise let $N$ be minimum such that $\varepsilon_N=3$. Define $\alpha=\alpha(\tilde{x})\in\R\setminus\Q$ by
\begin{equation*}
    \alpha=
    \begin{cases}
        [0;\varepsilon_{N+1},a_{N+2},\dots], & \text{if } \varepsilon_N=1, \varepsilon_{N+1}>0, \\
        [0;a_{N+3},a_{N+4},\dots], & \text{if } \varepsilon_N=1, \varepsilon_{N+1}=0, \\
        [0;3,2,a_{N+2},\dots], & \text{if }\varepsilon_N=3, \varepsilon_{N+1}=3, \\
        [0;3,3,a_{N+2},\dots], & \text{if }\varepsilon_N=3, \varepsilon_{N+1}=2, \\
        [0;3,1,a_{N+2},\dots], & \text{if }\varepsilon_N=3, \varepsilon_{N+1}=0.
    \end{cases}
\end{equation*}
Then for all $n\geq 1$
\begin{equation*}
    q_{n}\equiv \begin{cases}
        \varepsilon_{N+n}\pmod{4}, & \text{if $\varepsilon_N=1, \varepsilon_{N+1}>0$}, \\
        \varepsilon_{N+n+2}\pmod{4}, & \text{if $\varepsilon_N=1, \varepsilon_{N+1}=0$}, \\
        \varepsilon_{N+n-1}\pmod{4}, & \text{if $\varepsilon_N=3$},
    \end{cases}
\end{equation*}
where $p_n/q_n$ is the respective convergent of the $\alpha$ defined above. Therefore, by \eqref{eq:inequality_k_less_than_2}, we will have that
\begin{equation}\label{eq:10}
    k\left(\frac{1}{2}+i\alpha\right)\geq\limsup_{n\to\infty}\eta(n)=\ell(\tilde{x}).
\end{equation}
with equality if $k\left(\frac{1}{2}+i\alpha\right)>2$. In particular, for those $\tilde{x}\in\tilde{\Sigma}$ such that $\ell(\tilde{x})>2$, we will have that $\ell(\tilde{x})=k\left(\frac{1}{2}+i\alpha\right)\in L_{\frac{1}{2}+i\R}$. This finishes the proof of \eqref{eq:equality_after_t_2}.

Because of Schmidt's theorem (\Cref{thm:schmidt}), the set \[L_{\frac{1}{2}+i\R}\cap(-\infty,2)\subset L_{\Q(i)}\cap(-\infty,2),\] is countable. Moreover, also by Schmidt's theorem, we know that $\xi\in\C(i)\setminus\Q(i)$ satisfies $k(\xi)<2$, $k(\xi)\neq\sqrt{\frac{3}{5}\sqrt{41}}$ if and only if $\xi$ is equivalent to some $\theta_\Lambda$, that is, there are $a,b,c,d\in\Z[i]$ with $\abs{ad-bc}=1$ such that $\xi=(a\theta_\lambda+b)/(c\theta_\Lambda+d)$; similarly the same happens when $k(\xi)=\sqrt{\frac{3}{5}\sqrt{41}}$. In particular, $\xi$ belongs to an equivalence class that is countable. Since such equivalence classes are indexed by $\Lambda\in\{1, 5, 29,\dots\}$ except for the four equivalence classes corresponding to $k(\xi)=\sqrt{\frac{3}{5}\sqrt{41}}$, we conclude that there are at most countably many $\xi\in\C(i)\setminus\Q(i)$ with $k(\xi)<2$.  .

Given any $\ell(\tilde{x})=t\in L_{\varphi,f,\tilde{\Sigma}}\cap(-\infty,2)$, by the above discussion we can build $\alpha=\alpha(\tilde{x})\in\R\setminus\Q$ that satisfies \eqref{eq:10}. If $k\left(\frac{1}{2}+i\alpha\right)>2$, then $k\left(\frac{1}{2}+i\alpha\right)=\ell(\tilde{x})=t>2$, a contradiction. In particular we will have that $k\left(\frac{1}{2}+i\alpha\right)<2$ so such an $\alpha$ belongs to a countable set of equivalence classes. Since the set of $\tilde{x}\in\tilde{\Sigma}$ with the same corresponding $\alpha(\tilde{x})$ is also countable, we conclude that $L_{\varphi,f,\tilde{\Sigma}}\cap(-\infty,2)$ is countable.

In conclusion, since $L_{\frac{1}{2}+i\R}\cap(-\infty,2)$ is countable and $L_{\varphi,f,\tilde{\Sigma}}\cap(-\infty,2)$ is also countable, we have $\dim_H(L_{\frac{1}{2}+i\R}\cap (-\infty,t))=\dim_H(L_{\varphi,f,\tilde{\Sigma}}\cap(-\infty,t))=0$ for $t\leq 2$ and for $t>2$, the equality follows from \eqref{eq:equality_after_t_2}. 
\end{proof}

On the other hand, the set $L_{\varphi,f,\tilde{\Sigma}}\cap(-\infty,2)$ is different from $L_{\frac{1}{2}+i\R}\cap(-\infty,2)$. For example, the minimum of $L_{\frac{1}{2}+i\R}$ is $\sqrt{3}$ (because of \Cref{thm:schmidt}) but $\sqrt{2}\in L_{\varphi,f,\tilde{\Sigma}}\cap(-\infty,2)$. Indeed, if we consider the point $\tilde{x}=((a_n)_{n\in\Z},(1,0))$ where $a_{2k+1}=4$ and $a_{2k}=1$ for all $k\in\Z$, then for $m\geq 0$ we have $\varphi^{2m}(\tilde{x})=((a_n)_{n\in\Z},(1,0))$, so  
\begin{equation*}
    f(\varphi^{2m}(\tilde{x}))=[1;\overline{4,1}]+[0;\overline{4,1}]=1+2\left(\frac{1}{2}(\sqrt{2}-1)\right)=\sqrt{2}
\end{equation*}
and $\varphi^{2m+1}(\tilde{x})=((a_{n+1})_{n\in\Z},(0,1))$, so  
\begin{equation*}
    f(\varphi^{2m}(\tilde{x}))=\frac{[4;\overline{1,4}]+[0;\overline{1,4}]}{4}=\frac{4+2(2\sqrt{2}-2)}{4}=\sqrt{2}.
\end{equation*}

\subsection{An alternative formula for $k_{\frac{1}{2}}$}

    \begin{lemma}
        Let $\alpha\in\R\setminus\Q$. Then if $k_{\frac{1}{2}}\left(\frac{1}{2}+i\alpha\right)>2$, we have that $k_{\frac{1}{2}}\left(\frac{1}{2}+i\alpha\right)=\limsup_{n\to\infty} \kappa(n)$, where
        \begin{equation*}
         \kappa(n)=\begin{cases}
            \alpha_{n+1}+\beta_{n+1}, & \text{ if } p_n+q_n\equiv 0 \pmod{2} \text{ and } q_n\equiv 1\pmod{4}, \\
            \frac{\alpha_{n+1}+\beta_{n+1}}{2}, & \text{ else},
        \end{cases}   
        \end{equation*}
        and where $\frac{p_n}{q_n}$ are the convergents of $2\alpha$.
    \end{lemma}

    \begin{proof}
    Let $\alpha\in\R\setminus\Q$ and let $p,q\in\Z[i]$ be coprime with $\Re(\frac{p}{q})=\frac{1}{2}$. Write $\frac pq=\frac12+i\frac rs$, where $\gcd(r,s)=1$, and assume that \[\left|\frac12+i\alpha-\frac pq\right|<\frac{1}{2|q|^2}.\]
    
    We distinguish three cases according to the residue class of $s$ modulo $4$. Suppose first that $s\equiv2\pmod4$. Writing $s=2s_0$, we have $s_0$ odd and $r$ odd. Since $|s_0+ir|^2=s_0^2+r^2\equiv2\pmod4$, the Gaussian integer $s_0+ir$ is divisible by $1-i$, but not by $2$. Thus
    \[
    p=\frac{s_0+ir}{1-i},\qquad q=s_0(1+i),
    \]
    is the irreducible representation of $\frac12+i\frac{r}{2s_0}$. Consequently,
    \[
    |q|^2\left|\frac12+i\alpha-\frac pq\right|
    =2s_0^2\left|\alpha-\frac{r}{2s_0}\right|
    =|s_0|\,|(2\alpha)s_0-r|
    <\frac12.
    \]
    By Legendre's theorem, $r/s_0=p_n/q_n$ is a convergent of $2\alpha$. Since $p_n+q_n$ is even and $q_n=s_0\equiv1\pmod4$, we obtain
    \[
    |q|^{-2}\left|\frac12+i\alpha-\frac pq\right|^{-1}
    =\alpha_{n+1}+\beta_{n+1}.
    \]
    
    Next suppose that $s\equiv0\pmod4$. Then $s=2s_0$ with $s_0$ even, while $r$ is odd. Since $|s_0+ir|^2=s_0^2+r^2\equiv1\pmod4,$ the fraction is already in lowest terms: $p=s_0+ir,\qquad q=2s_0$. Hence
    \[
    |q|^2\left|\frac12+i\alpha-\frac pq\right|
    =4s_0^2\left|\alpha-\frac{r}{2s_0}\right|
    =2|s_0|\,|s_0(2\alpha)-r|
    <\frac12.
    \]
    Again Legendre's theorem gives $r/s_0=p_n/q_n$, and now $p_n+q_n$ is odd. Therefore,
    \[
    |q|^{-2}\left|\frac12+i\alpha-\frac pq\right|^{-1}
    =\frac{\alpha_{n+1}+\beta_{n+1}}{2}.
    \]
    
    Finally, suppose that $s$ is odd. Writing $\frac rs=\frac{2r}{2s}=\frac{r_0}{2s}$, where $r_0=2r$, we have
    \[
    \frac12+i\frac rs=\frac{s+ir_0}{2s}.
    \]
    Since $|s+ir_0|^2=s^2+r_0^2\equiv1\pmod4$, the Gaussian integers $p=s+ir_0$ and $q=2s$ are coprime. Thus
    \[
    |q|^2\left|\frac12+i\alpha-\frac pq\right|
    =2|s|\,|s(2\alpha)-r|
    <\frac12.
    \]
    Legendre's theorem again shows that $r/s=p_n/q_n$ is a convergent of $2\alpha$, and therefore
    \[
    |q|^{-2}\left|\frac12+i\alpha-\frac pq\right|^{-1}
    =\frac{\alpha_{n+1}+\beta_{n+1}}{2}.
    \]
    
    This proves the claimed formula for $k_{\frac12}\!\left(\frac12+i\alpha\right)$.
    \end{proof}

\section{Reduction of $L_{\frac{1}{2}+i\R}$ to the dynamical Lagrange spectrum associated to a horseshoe and continuity of dimension}\label{sec:fractal_properties}

To prove the results about $L_{\frac{1}{2}+i\R}$, we are going to construct a dynamical representation of $L_{\frac{1}{2}+i\R}$ in a smooth setting, similar to what is done in the case of the classical spectrum. With this, we are going to reduce the study of the interesting parts of $L_{\frac{1}{2}+i\R}$ to a dynamical Lagrange spectrum associated to a horseshoe.

\subsection{The map $\tilde{T}_1$}\label{subsec:the_map_T_1}
Consider the smooth map $\tilde{T}_1:(0,1)^2\times\mathcal{I}\to[0,1)^2\times\mathcal{I}$ given by
\begin{multline*}
    \tilde{T}_1\left((x,y),(\varepsilon_0,\varepsilon_1)\right)=\left(\left(\left\{\frac1{x}\right\},\frac1{y+\lfloor 1/x\rfloor}\right)\right., \\
    \left.\left(\varepsilon_1,\left\lfloor\left\{\frac1{x}\right\}\right\rfloor\varepsilon_1+\varepsilon_0 \pmod{4}\right)\right).
\end{multline*}

Let us see that, similar to the map $T_1$
used in the classical Lagrange spectrum, the map $\tilde{T}_1$ preserves an invariant smooth density. In the construction of the classical spectrum \cite{ArnouxFlotGeodesique}, one uses $T_1:(0,1)\times(0,1)\to [0,1)\times(0,1)$ by
\begin{equation*}
    T_1(x,y)=\left(\left\{\frac1{x}\right\},\frac1{y+\lfloor 1/x\rfloor}\right)
\end{equation*}
and $T:S\to\overline{S}$ where $S=\{(x,y)\in\R^2:0<x<1,0<y<1/(1+x)\}$ by
\begin{equation*}
    T(x,y)=\left(\left\{\frac1{x}\right\},x-x^2y\right).
\end{equation*}
In particular
\begin{equation*}
    T_1([0;a_0,a_1,\dots],[0;b_1,b_2,\dots])=([0;a_1,a_2,\dots],[0;a_0,b_1,b_2,\dots]).
\end{equation*}

Define the real valued map $\tilde{f}:(0,1)^2\times\mathcal{I}\to\R$ by 
\begin{equation*}
    \tilde{f}\left((x,y),(\varepsilon_0,\varepsilon_1)\right)=\begin{cases}
        1/x+y, & \text{if $\varepsilon_1=0$}, \\
        2(1/x+y), & \text{if $\varepsilon_1=2$}, \\
        \frac{1/x+y}{4}, & \shortstack[l]{if $\varepsilon_1$ odd \\ and $\varepsilon_0$ even}, \\
        \max\left\{\frac{1/x+y}{4},2\left(\frac{1}{1+y}+\frac{1}{1/x-1}\right)\right\}, &\shortstack[l]{if $\varepsilon_1,\varepsilon_0$ odd \\ and $\varepsilon_0=\varepsilon_1$},  \\
        \max\left\{\frac{1/x+y}{4},2\left(\frac{1}{1-y}-\frac{1}{1/x+1}\right)\right\}, &\shortstack[l]{if $\varepsilon_1,\varepsilon_0$ odd \\ and $\varepsilon_0\neq\varepsilon_1$},  \\
    \end{cases}
\end{equation*}

If $h:\overline{S}\to[0,1]^2$ is given by $h(x,y)=(x,y/(1-xy))$, then $h$ is a conjugation between $T$ and $T_1$. It is known that $T$ preserves the Lebesgue measure $\Leb$, because $T$ is a bijection (up to a null set) and has derivative
\begin{equation*}
    \begin{pmatrix}
        -1/x^2 & 0 \\
        1-2xy & -x^2
    \end{pmatrix}
\end{equation*}
which has determinant 1. Thus, $T_1$ preserves the smooth measure $h_{*}(\Leb)$, which will be given by the density $d\mu=\frac{1}{(1+xy)^2}dxdy$. Similarly, we can define the map $\tilde{T}:S\times\mathcal{I}\to\overline{S}\times\mathcal{I}$ by
\begin{equation*}
    \tilde{T}\left((x,y),(\varepsilon_0,\varepsilon_1)\right)=\left(T(x,y),\left(\varepsilon_1,\left\lfloor\left\{\frac1{x}\right\}\right\rfloor\varepsilon_1+\varepsilon_0 \pmod{4}\right)\right).
\end{equation*}
Using the same argument as before, we can show that $\tilde{T}$ preserves Lebesgue measure and by using the conjugation $\tilde{h}((x,y),(\varepsilon_0,\varepsilon_1))=(h(x,y),(\varepsilon_0,\varepsilon_1))$ we have that $\tilde{T}_1$ preserves the smooth measure $\tilde{h}_{*}(\Leb)$. 

\subsection{Reduction to a horseshoe}

By \eqref{eq:equal_dimensions}, it suffices to prove that the function
\[
t\in\R\longmapsto \dim_H\left(L_{\varphi,f,\tilde{\Sigma}}\cap(-\infty,t)\right)\in[0,1]
\]
is continuous. To this end, we shall apply the continuity theorem for dynamical spectra associated with horseshoes (\Cref{thm:continuity}).

Given $N\in\N_{>0}$, observe that $\tilde{\Lambda}(N)=C(N)\times C(N)\times\mathcal{I}$ is a maximal invariant set of $\tilde{T}_1$. Moreover, it is the product of the Gauss--Cantor set $C(N)$ and the dynamically defined Cantor set $C(N)\times\mathcal{I}$. Since the derivative of $\tilde{T}_1$ is hyperbolic on $\tilde{\Lambda}(N)$, it follows that $(\tilde{T}_1|_{\tilde{\Lambda}(N)},\tilde{\Lambda}(N))$ is a horseshoe.

The local stable and unstable manifolds of this horseshoe are segments of vertical and horizontal lines. Since these are everywhere transverse to the gradient of the functions defining $\tilde{f}$, we conclude that $\tilde{f}\in\mathcal{R}_{\varphi,\tilde{\Lambda}(N)}$, where $\mathcal{R}_{\varphi,\tilde{\Lambda}(N)}$ is defined in \eqref{eq:R_varphi_Lambda}.

It remains to verify the dimension formula hypothesis in \Cref{thm:continuity}. Every subhorseshoe of $\tilde{\Lambda}(N)$ decomposes, by the spectral theorem, into subhorseshoes that are products of stable and unstable Cantor sets. One of these Cantor sets is precisely a Gauss--Cantor set which is non-essentially affine by \Cref{prop:geometric_properties_prop1}. Therefore, \Cref{thm:gugu_dimension_formula} yields the required dimension formula for every subhorseshoe. Hence all the hypotheses of \Cref{thm:continuity} are satisfied.

Because of the formula for the height function $\tilde{f}$, we can easily see that if $(x,y)\in C(N)\times C(N)$, then $\tilde{f}((x,y),(\varepsilon_0,\varepsilon_1))\leq 4N$. As it is a well known fact that $\dim_H(C(N))>\frac{1}{2}$ for all $N\geq 2$, we have $\dim_H(\tilde{\Lambda}(N))>1$, and by \Cref{prop:dimension_of_subimages}, we get $d_0(4N)=\dim_H(L_{\varphi,f,\tilde{\Sigma}}\cap (-\infty,4N))=1$. In particular, choosing $N=2$, the function $d_0(t)=\dim_H(L_{\varphi,f,\tilde{\Sigma}}\cap (-\infty,t))$ is continuous for all $t\geq 8$. On the other hand, if the continued fraction of $x=[a_0;a_1,\dots]$ starts with some $a_0\geq 16N$, we get that $\tilde{f}((x,y),(\varepsilon_0,\varepsilon_1))>4N$. Therefore, points with Lagrange values at most $4N$ can only have finitely many coefficients $a_n\geq 16N$. A consequence of this is that they must have the same Lagrange value as some point in $\tilde{\Lambda}(16N-1)$ and, conversely, any Lagrange value of points in $\tilde{\Lambda}(16N-1)$ is an element of $L_{\varphi,f,\tilde{\Sigma}}$. As consequence, using \eqref{eq:equality_after_t_2}, we have for all $t\leq 4N$
\begin{gather}
    L_{\varphi,f,\tilde{\Sigma}}\cap (-\infty,t]=L_{\tilde{T}_1,\tilde{f},\tilde{\Lambda}(16N-1)}\cap (-\infty,t], \nonumber \\
    L_{\frac{1}{2}+i\R}\cap[2,t]=L_{\varphi,f,\tilde{\Sigma}}\cap [2,t]=L_{\tilde{T}_1,\tilde{\Lambda}(16N-1),\tilde{f}}\cap [2,t]. \label{eq:equality_between_3_spectra}
\end{gather}
Hence  $\dim_H(L_{\varphi,f,\tilde{\Sigma}}\cap (-\infty,t))=\dim_H(L_{\tilde{T}_1,\tilde{f},\tilde{\Lambda}(31)}\cap (-\infty,t))$ for $t\leq 8$. We can now extract \Cref{thm:main}, \Cref{itm:thmcont} and \Cref{itm:thmcont2} as a consequence of \Cref{thm:continuity}. 

The argument to prove that $L_{\frac{1}{2}+i\R}$ is closed is the following. Since the map $\tilde{f}$ is continuous, we have that $L_{\tilde{T}_1,\tilde{f},\tilde{\Lambda}(16N-1)}\cap [2,4N]$ is closed, so by \eqref{eq:equality_between_3_spectra} one has that $L_{\frac{1}{2}+i\R}\cap[2,4N]$ is also closed, which shows that the whole $L_{\frac{1}{2}+i\R}$ is closed. Similarly, since $\tilde{f}$ is the maximum of smooth functions which have transverse gradient to the vertical and horizontal directions, that is, transverse to the stable and unstable manifolds of $\tilde{\Lambda}(16N-1)$. By \Cref{thm:perfect_set} we have that $L_{\tilde{T}_1,\tilde{f},\tilde{\Lambda}(16N-1)}^{\prime}=L_{\tilde{T}_1,\tilde{\Lambda}(16N-1),\tilde{f}}^{\prime\prime}$, so by \eqref{eq:equality_between_3_spectra} we have  $L_{\frac{1}{2}+i\R}^{\prime}\cap[2,4N]=L_{\frac{1}{2}+i\R}^{\prime\prime}\cap[2,4N]$ so taking $N$ to infinity proves that $L_{\frac{1}{2}+i\R}$ is perfect. Moreover, since $L_{\varphi,f,\tilde{\Sigma}}\cap(-\infty,2)$ is discrete, \Cref{thm:perfect_set} also implies that $\max\{t\in\R:\dim_H(L_{\tilde{T}_1,\tilde{f},\tilde{\Lambda}(16N-1)}\cap(-\infty,t))=0\}=2$, so $\dim_H(L_{\tilde{T}_1,\tilde{\Lambda}(31),\tilde{f}}\cap(2,2+\varepsilon))>0$ for all $\varepsilon>0$. Therefore by \eqref{eq:equality_between_3_spectra} we get that $\dim_H(L_{\frac{1}{2}+i\R}\cap[2,2+\varepsilon])=L_{\tilde{T}_1,\tilde{f},\tilde{\Lambda}(31)}\cap [2,2+\varepsilon])>0$. Since $L_{\frac{1}{2}+i\R}\cap(-\infty,2)$ is discrete this finishes the proof of \Cref{thm:main}.

\subsection{Alternative (explicit) proof that $d_0(2+\varepsilon)>0$}

The fact that $d_0(2+\varepsilon)>0$ is a direct consequence of \Cref{thm:perfect_set}. In this section, we will build explicit dynamically defined Cantor sets contained in $L_{\frac{1}{2}+i\R}\cap(-\infty,2+\varepsilon)$. These Cantor sets will be obtained as images by $f$ of a certain bi-infinite subshifts we are going to describe now.

The basis for our construction is the fact that the real quadratic irrationals $x_k=[0;\overline{(4,1)^k,1,3,1}]$ are such that $k(\frac{1}{2}+ix_k)$ forms a decreasing sequence converging to 2. 

\begin{lemma}
    Given $k\geq 2$, define $x_k=[0;\overline{(4,1)^k,1,3,1}]$. Then, we have that
    \begin{multline}\label{eq:x_k_value}
        k\left(\frac{1}{2}+ix_k\right)=\max\big([1;\overline{1,3,1,(4,1)^k}]+[0;\overline{(4,1)^k,3,1,1}]), \\
        2([0;1,\overline{1,(1,4)^k,1,3}]+[0;2,\overline{1,(4,1)^k,1,3}])\big).
    \end{multline}
    In particular
    \begin{equation}\label{eq:x_k_value_bound}
        2<k\left(\frac{1}{2}+ix_k\right)<2+\frac{1}{2^{2k-2}}.
    \end{equation}
\end{lemma}

\begin{proof}
    First, let us notice that because of the periodicity of the continued fraction of $x_k$, we actually have $k(\frac{1}{2}+ix_k)=\ell(\tilde{x}_k)$, where $\tilde{x}_k=((a_n)_{n\in\Z},(1,0))$ and $(a_n)_{n\in\Z}$ is the periodic sequence with period $(a_{1},\dots,a_{2k+3})=((4,1)^k,1,3,1)$. It is easy to see that if $q_m$ denotes the $m$--th convergent of the quadratic irrational $x_k$, then the sequence $(q_m \pmod{4})_{m\geq 0}$ is periodic with period $(1,(0,1)^k,1,0)$. Let us denote $\varepsilon_m\equiv q_m\pmod{4}$ for $m\geq 0$ where $0\leq\varepsilon_m<4$. We want to compute the heights $f(\varphi^m(\tilde{x}_k))$ for $0\leq m< 2k+3$. Note that
    \begin{equation*}
        \varphi^{m}(\tilde{x}_k)=\big((a_{n+m})_{n\in\Z},(\varepsilon_{m},\varepsilon_{m+1})\big).
    \end{equation*}
    
    We have to consider several cases. 
    \begin{itemize}
        \item  If $m=2s$ with $0\leq s< k-1$, then $(\varepsilon_m,\varepsilon_{m+1})=(1,0)$, so 
        \begin{align*}
            f(\varphi^{m}(\tilde{x}_k))=\alpha_{2s+2}+\beta_{2s+2}&=[1;(4,1)^{k-s-1},1,3,1,\dots]+[0;(4,1)^{s+1},3,1,\dots]\\
            &<[1;4]+[0;4]=\frac{3}{2}.
        \end{align*}
        \item If $m=2s+1$ with $0\leq s\leq k-1$ or $m=2k+2$, then $(\varepsilon_m,\varepsilon_{m+1})=(0,1)$, and since $a_n\leq 4$, we have
        \begin{equation*}
            f(\varphi^m(\tilde{x}_k))=\frac{\alpha_{m+2}+\beta_{m+2}}{4}<\frac{6}{4}=\frac{3}{2}.
        \end{equation*}
        \item If $m=2k+1$ then $(\varepsilon_{2k+1},\varepsilon_{2k+2})=(1,0)$, so
        \begin{align*}
            f(\varphi^{2k+1}(\tilde{x}_k))=\alpha_{2k+3}+\beta_{2k+3}&=[1;(4,1)^{k},1,3,1,\dots]+[0;3,1,(1,4)^{k},\dots]\\
            &<[1;4]+[0;3]=\frac{19}{12}.
        \end{align*}
        \item If $m=2k$, then $(\varepsilon_{2k},\varepsilon_{2k+1})=(1,1)$, so
        \begin{equation*}
            f(\varphi^{2k}(\tilde{x}_k))=\max\left(\frac{\alpha_{2k+2}+\beta_{2k+2}}{4},2\left(\frac{1}{1+\beta_{2k+2}}+\frac{1}{\alpha_{2k+2}-1}\right)\right)
        \end{equation*}
        where $\alpha_{2k+2}=[3;\overline{1,(4,1)^k,1,3}]$ and $\beta_{2k+2}=[0;\overline{1,(1,4)^k,1,3}]$. We have
        \begin{align*}
            \frac{1}{1+\beta_{2k+2}}+\frac{1}{\alpha_{2k+2}-1}&=[0;1,\overline{1,(1,4)^k,1,3}]+[0;2,\overline{1,(4,1)^k,1,3}] \\
            &>[0;1,1,1,4]+[0;2,1,4,1]=\frac{237}{238}>\frac{19}{24}>\frac{3}{4},
        \end{align*}
        so $f(\varphi^{2k}(\tilde{x}_k))=2([0;1,\overline{1,(1,4)^k,1,3}]+[0;2,\overline{1,(4,1)^k,1,3}])$.
        \item If $m=2k-2$ then $(\varepsilon_{2k-2},\varepsilon_{2k-1})=(1,0)$, so 
        \begin{equation*}
            f(\varphi^{2k-2}(\tilde{x}_k))=\alpha_{2k}+\beta_{2k}=[1;\overline{1,3,1,(4,1)^k}]+[0;\overline{(4,1)^k,3,1,1}].
        \end{equation*}
    \end{itemize}

    This proves \eqref{eq:x_k_value}. Finally we will use the elementary identity $[0;1+u,v]+[0;1,u,v]=1$ to prove \eqref{eq:x_k_value_bound}. Indeed, 
    \begin{multline*}
        [1;\overline{1,3,1,(4,1)^k}]+[0;\overline{(4,1)^k,3,1,1}]) \\
        <[1;1,3,1,\overline{4,1}]+[0;4,1,\overline{4,1}]+\frac{1}{2^{2k-2}}=2+\frac{1}{2^{2k-2}},
    \end{multline*}
    and
    \begin{multline*}
        2([0;1,\overline{1,(1,4)^k,1,3}]+[0;2,\overline{1,(4,1)^k,1,3}]) \\
        <2([0;1,1,\overline{1,4}]+[0;2,\overline{1,4}]+\frac{1}{2^{2k-1}})=2+\frac{1}{2^{2k-2}}.
    \end{multline*}
\end{proof}

In virtue of the previous lemma, we can define the following bi-infinite  Bernoulli shift 
\begin{equation*}
    \tilde{\Sigma}_k:=\left\{\big((\gamma_i)_{i\in\Z},(1,0)\big) :\gamma_i\in\{((4,1)^k,1,3,1),((4,1)^{k+1},1,3,1)\},\forall i\in\Z\right\}.
\end{equation*}
Because of the proof of the previous lemma and \eqref{eq:x_k_value_bound}, we have that $\ell(\tilde{x})<2+2^{-2k+2}$ for all $\tilde{x}\in\tilde{\Sigma}_k$.

In particular, given any $\varepsilon>0$, we let $k$ be so large so that $2^{-2k+2}<\varepsilon$ and thus
\begin{equation*}
    \ell_{\varphi,f}(\tilde{\Sigma}_k)\subset L_{\varphi,f,\tilde{\Sigma}}\cap(2,2+\varepsilon)=L_{\frac{1}{2}+i\R}\cap(2,2+\varepsilon).
\end{equation*}
By \Cref{prop:dimension_of_subimages} we have that $\dim_H(\ell_{\varphi,f}{\Sigma}_k)=\min\{1,2\cdot \dim_H(K)\}=2\cdot \dim_H(K)$ where $K$ is the regular Gauss-Cantor set
\begin{equation*}
     K =\{[0;(4,1)^{k_1},1,3,1,(4,1)^{k_2},1,3,1,\dots]:k_i\in\{k,k+1\},\forall i\geq0\}.
\end{equation*}

Therefore
\begin{equation*}
    d_0(2+\varepsilon)=\dim_H(L_{\frac{1}{2}+i\R}\cap(2,2+\varepsilon))\geq \dim_H(\ell_{\varphi,f}(\tilde{\Sigma}_k))=2\cdot \dim_H(K)>0.
\end{equation*}

\begin{remark}
This computation gives a explicit lower estimate for the modulus of continuity of $d_0(2+\varepsilon)$. 
\end{remark}

\bibliographystyle{plain}
\bibliography{bibliography}

\begin{thebibliography}{10}

\bibitem{ArnouxFlotGeodesique}
Pierre Arnoux.
\newblock Le codage du flot g\'eod\'esique sur la surface modulaire.
\newblock {\em Enseign. Math. (2)}, 40(1-2):29--48, 1994.

\bibitem{Bombieri}
Enrico Bombieri.
\newblock Continued fractions and the {M}arkoff tree.
\newblock {\em Expo. Math.}, 25(3):187--213, 2007.

\bibitem{Bug}
Yann Bugeaud, Gerardo Gonz\'alez~Robert, and Mumtaz Hussain.
\newblock Metrical properties of {H}urwitz continued fractions.
\newblock {\em Adv. Math.}, 468:Paper No. 110208, 94, 2025.

\bibitem{Cassels}
J.~W.~S. Cassels.
\newblock {\em An introduction to {D}iophantine approximation}, volume No. 45
  of {\em Cambridge Tracts in Mathematics and Mathematical Physics}.
\newblock Cambridge University Press, New York, 1957.

\bibitem{CerqueiraMatheusGugu}
Aline Cerqueira, Carlos Matheus, and Carlos~Gustavo Moreira.
\newblock Continuity of {H}ausdorff dimension across generic dynamical
  {L}agrange and {M}arkov spectra.
\newblock {\em J. Mod. Dyn.}, 12:151--174, 2018.

\bibitem{SecondLagrangeSpectra}
Hao Cheng, Harold Erazo, Carlos~Gustavo Moreira, and Thiago Vasconcelos.
\newblock On the geometry of the second lagrange spectra, 2026.
\newblock ArXiv:2602.09228.

\bibitem{SandoelThesis}
Sandoel de~Brito~Vieira.
\newblock {\em Markov and Lagrange spectra}.
\newblock PhD thesis, Instituto de Matem\'{a}tica Pura e Aplicada, 2020.

\bibitem{Florek}
Jan Florek.
\newblock Roots of markoff quadratic forms as strongly badly approximable
  numbers, 2011.
\newblock ArXiv:1106.1844.

\bibitem{Ford}
Lester~R. Ford.
\newblock On the closeness of approach of complex rational fractions to a
  complex irrational number.
\newblock {\em Trans. Amer. Math. Soc.}, 27(2):146--154, 1925.

\bibitem{Gurwood}
C.~Gurwood.
\newblock {\em Diophantine approximation and the Markov chain}.
\newblock PhD thesis, New York University, 1976.

\bibitem{Harcos}
G.~Harcos.
\newblock Milyen távolságokra eshet egy valós szám az összes racionális
  számtól? [what are the possible distances between a real number and all the
  rationals?], 1996.
\newblock Undergraduate thesis, Eötvös Loránd University.

\bibitem{Hur}
A.~Hurwitz.
\newblock \"uber die {E}ntwicklung complexer {G}r\"ossen in {K}ettenbr\"uche.
\newblock {\em Acta Math.}, 11(1-4):187--200, 1887.

\bibitem{SergeLang}
Serge Lang.
\newblock {\em Introduction to {D}iophantine approximations}.
\newblock Springer-Verlag, New York, second edition, 1995.

\bibitem{Phasetransition}
D.~Lima and C.~G. Moreira.
\newblock Phase transitions on the {M}arkov and {L}agrange dynamical spectra.
\newblock {\em Ann. Inst. H. Poincar\'{e} C Anal. Non Lin\'{e}aire},
  38(5):1429--1459, 2021.

\bibitem{ML}
Davi Lima and Carlos~Gustavo Moreira.
\newblock Dynamical characterization of initial segments of the {M}arkov and
  {L}agrange spectra.
\newblock {\em Monatsh. Math.}, 199(4):817--852, 2022.

\bibitem{LMV}
Davi Lima, Carlos~Gustavo Moreira, and Christian Camilo~Silva Villamil.
\newblock Continuity of fractal dimensions in conservative generic markov and
  lagrange dynamical spectra, 2025.

\bibitem{M1}
Carlos~Gustavo Moreira.
\newblock Geometric properties of the {M}arkov and {L}agrange spectra.
\newblock {\em Ann. of Math. (2)}, 188(1):145--170, 2018.

\bibitem{CartesianProduct}
Carlos~Gustavo Moreira.
\newblock Geometric properties of images of cartesian products of regular
  {C}antor sets by differentiable real maps.
\newblock {\em Math. Z.}, 303(1):Paper No. 3, 19, 2023.

\bibitem{NIT}
Hitoshi Nakada, Shunji Ito, and Shigeru Tanaka.
\newblock On the invariant measure for the transformations associated with some
  real continued-fractions.
\newblock {\em Keio Engrg. Rep.}, 30(13):159--175, 1977.

\bibitem{PT}
Jacob Palis and Floris Takens.
\newblock {\em Hyperbolicity and sensitive chaotic dynamics at homoclinic
  bifurcations}, volume~35 of {\em Cambridge Studies in Advanced Mathematics}.
\newblock Cambridge University Press, Cambridge, 1993.
\newblock Fractal dimensions and infinitely many attractors.

\bibitem{Perron}
Oskar Perron.
\newblock {\"U}ber die approximation irrationaler zahlen durch rationale. ii.
\newblock {\em Sitzungsber. Heidelb. Akad. Wiss.}, (8), 1921.

\bibitem{SergioGuguETDS}
Sergio~Augusto Roma\~na Ibarra and Carlos Gustavo~T. de~A.~Moreira.
\newblock On the {L}agrange and {M}arkov dynamical spectra.
\newblock {\em Ergodic Theory Dynam. Systems}, 37(5):1570--1591, 2017.

\bibitem{Sch}
Asmus~L. Schmidt.
\newblock Diophantine approximation of complex numbers.
\newblock {\em Acta Math.}, 134:1--85, 1975.

\bibitem{SchmidtCforms}
Asmus~L. Schmidt.
\newblock On {$C$}-minimal forms.
\newblock {\em Math. Ann.}, 215:203--214, 1975.

\bibitem{Vulakh}
L.~Ja. Vulah.
\newblock The {M}arkov spectrum of imaginary quadratic fields {$Q(i\surd D)$},
  where {$D\not\equiv 3({\rm mod}\ 4)$}.
\newblock {\em Vestnik Moskov. Univ. Ser. I Mat. Meh.}, 26(6):32--41, 1971.

\bibitem{Zagier}
Don Zagier.
\newblock On the number of {M}arkoff numbers below a given bound.
\newblock {\em Math. Comp.}, 39(160):709--723, 1982.

\end{thebibliography}

\end{document}